\documentclass[11pt]%{ip-journal}
{amsart}
\input epsf
\usepackage{latexsym}
\usepackage{amssymb,amsmath,amsthm,amscd, amssymb,
graphicx, enumerate, verbatim, mathrsfs, color, url}
\usepackage[all]{xy}

\newcommand{\Q}{{\mathbb Q}}
\def\Diff{\operatorname{Diff}}

\def\Int{\operatorname{Int}}

\numberwithin{equation}{section}
\newtheorem{cor}[equation]{Corollary}
\newtheorem{lem}[equation]{Lemma}
\newtheorem{prop}[equation]{Proposition}
\newtheorem{thm}[equation]{Theorem}

\newtheorem{Example}[equation]{Example}
\newenvironment{ex}{\begin{Example}\rm}{\end{Example}}
\newtheorem{remark}[equation]{Remark}
\newenvironment{rmk}{\begin{remark}\rm}{\end{remark}}
\def\co{\colon\thinspace}

\newcommand{\im}{\ensuremath{\operatorname{Im}}}

\newcommand{\e}{\varepsilon}

\def\a{\alpha}

\def\g{\gamma}

\def\d{\partial}

\def\r{\rho}

\def\i{\iota}
\def\s{\sigma}

\def\S1{\bf S^1}

\makeatletter
\def\equalsfill{$\m@th\mathord=\mkern-7mu
\cleaders\hbox{$\!\mathord=\!$}\hfill
\mkern-7mu\mathord=$}
\makeatother

\begin{document}

\abovedisplayskip=6pt plus3pt minus3pt
\belowdisplayskip=6pt plus3pt minus3pt

\title[Nonnegatively curved metrics and pseudoisotopies]
{\bf Space of nonnegatively curved metrics and pseudoisotopies}

\thanks{
%53C20 Global Riemannian geometry, including pinching
%57N37 Isotopy and pseudo-isotopy 
%19D10 Algebraic K-theory of spaces 
%18F25 Algebraic $K$-theory and $L$-theory
%55P62 Rational homotopy theory 
}\rm

\author{Igor Belegradek\and F. Thomas Farrell\and Vitali Kapovitch}

\address{Igor Belegradek\\School of Mathematics\\ Georgia Institute of
Technology\\Atlanta, GA, USA 30332-0160\\}\email{ib@math.gatech.edu}

\address{F. Thomas Farrell\\Department of Mathematical Sciences\\
Binghamton University\\PO Box 6000, Binghamton, New York, USA 13902-6000\\}\email{farrell@math.binghamton.edu}

\address{Vitali Kapovitch\\Department of Mathematics\\
University of Toronto\\
Toronto, ON, Canada M5S 2E4\\}\email{vtk@math.toronto.edu}

%\thanks{}

\date{}
\begin{abstract}
Let $V$ be an open manifold with complete
nonnegatively curved metric such that the normal 
sphere bundle to a soul has no section.
We prove that the souls of nearby nonnegatively 
curved metrics on $V$ are smoothly close. Combining 
this result with some topological properties of pseudoisotopies 
we show that for many $V$ the space of complete nonnegatively
curved metrics has infinite higher homotopy groups.
\end{abstract}
\maketitle
%\tableofcontents

\section{Introduction}

Throughout the paper ``smooth'' means $C^\infty$, all manifolds are smooth, and
any set of smooth maps, such as
diffeomorphisms, embeddings, pseudoisotopies, or Riemannian metrics,
is equipped with the smooth compact-open topology.  

Let ${\mathfrak R}_{K\ge 0}(V)$ denote the space of complete Riemannian 
metrics of nonnegative sectional curvature on a connected manifold $V$\!. 
The group $\Diff V$ acts on ${\mathfrak R}_{K\ge 0}(V)$ by
pullback.
% with the orbit map of the metric $h$ given by $\theta_h(\phi)=\phi^{-1*}h$.
Let ${\mathfrak M}_{K\ge 0}(V)$ be the associated {\it moduli space}, 
the quotient space of ${\mathfrak R}_{K\ge 0}(V)$
by the above $\Diff V$ action.

Many open manifolds $V$ for which ${\mathfrak M}_{K\ge 0}(V)$ 
is not path-connected, or even has infinitely many path-components, were
found in~\cite{KPT, BKS-mod1, BKS-mod2, Ott-pairs}.  
On the other hand, it was shown in~\cite{BelHu-modr2} that ${\mathfrak R}_{K\ge 0}(\mathbb R^2)$ is 
homeomorphic to the separable Hilbert space, and 
the associated moduli space ${\mathfrak M}_{K\ge 0}(\mathbb R^2)$ 
cannot be separated by a closed
subset of finite covering dimension.

Recall that any open complete
manifold $V$ of $K\ge 0$ contains a compact totally convex submanifold
without boundary, called a {\it soul\,}, such that
$V$ is diffeomorphic to the interior of a tubular neighborhood
of the soul~\cite{CheGro}. 
We call a connected open manifold {\it indecomposable\,} 
if it admits a complete metric of $K\ge 0$ such that the normal sphere bundle
to a soul has no section.

Let $N$ be a compact manifold (e.g. a tubular neighborhood of a soul).
A key object in this paper is the map
$\iota_N\co P(\d N)\to\Diff N$  
that extends a pseudoisotopy from a fixed collar neighborhood of 
$\d N$ to a diffeomorphism of $N$ supported in the collar neighborhood.
Here $P(\d N)$ and $\Diff N$ are the topological groups of pseudoisotopies of $\d N$ and
diffeomorphisms of $N$, respectively, see Section~\ref{sec: pi} for background. 
Let $\pi_j(\iota_N)$ be the homomorphism induced by $\iota_N$ 
on the $j$th homotopy group based at the identity. 
We prove the following.

\begin{thm} 
\label{thm-intro: main}
Let $N$ be a compact manifold with indecomposable interior. 
Then for every $h\in {\mathfrak R}_{K\ge 0}(\Int N)$ and each 
$k\ge 2$, the group $\ker\pi_{k-1}(\iota_N)$ is a quotient of a
subgroup of $\pi_k ({\mathfrak R}_{K\ge 0}(\Int N), h)$. 
\end{thm}

Prior to this result there has been no tool to detect nontrivial
higher homotopy groups of ${\mathfrak R}_{K\ge 0}(V)$.

We make a systematic study of $\ker\pi_j(\iota_N)$ and
find a number of manifolds for which $\ker\pi_j(\iota_N)$ is infinite
and $\Int N$ admits a complete metric of $K\ge 0$. 
Here is a sample of  what we can do:
 
\begin{thm} \label{thm-intro: main app}
Let $U$ be the total space of one of the following vector bundles:
\vspace{-3pt}
\begin{enumerate}
\item[\textup{(1)}]
the tangent bundle to $S^{2d}$, $CP^d$, $HP^d$, $d\ge 2$, and the Cayley plane, \vspace{1pt}
\item[\textup{(2)}] 
the Hopf $\,\mathbb R^4$ or $\mathbb R^3$ bundle over $HP^d$, $d\ge 1$,\vspace{1pt}
\item[\textup{(3)}] 
any linear $\mathbb R^4$ bundle over $S^4$ with nonzero Euler class, \vspace{1pt}
\item[\textup{(4)}] any nontrivial $\mathbb R^3$ bundle over $S^4$,\vspace{1pt}
\item[\textup{(5)}] the product of any bundle in 
\textup{(1)}, \textup{(2)}, \textup{(3)}, \textup{(4)} and 
any closed manifold of $K\ge 0$ and nonzero Euler characteristic.
\end{enumerate}
\vspace{-3pt}
Then there exists $m$ such that every path-component of ${\mathfrak R}_{K\ge 0}(U\times S^m)$ 
has some nonzero rational homotopy group.
\end{thm}

It is well-known that each $U$ in Theorem~\ref{thm-intro: main app}
admits a complete metric of $K\ge 0$:
For bundles in (3), (4) this follows from~\cite{GroZil}, and the bundles in (1), (2)
come with the the standard Riemannian submersion metrics, see Example~\ref{ex: indecomp}\,(2).

We can also add to the list in Theorem~\ref{thm-intro: main app}
some $\mathbb R^4$ and $\mathbb R^3$ bundles over $S^5$ and $S^7$ 
and an infinite family of $\mathbb R^3$ bundles over $CP^2$, which admit a complete
metrics of $K\ge 0$ thanks to~\cite{GroZil} and~\cite{GroZil-lift}, respectively. 
Other computations are surely possible. In fact we are yet to find $N$
with indecomposable interior and such that $\i_N$ is injective on all
homotopy groups; the latter does happen when $N=D^n$,
see Remark~\ref{rmk: iota Dn}.

\renewcommand{\thefootnote}{*}
We are unable\footnote{Explicit computations of $m$ will appear in the
upcoming work of Mauricio Bustamante, the second author and Jiang Yi.}to compute $m$ in Theorem~\ref{thm-intro: main app}. 
Given $U$ we find $k\ge 1$ such that for every $l\gg k$ 
there is $\s\in \{0,1,2,3\}$ for which the group
$\pi_k \, {\mathfrak R}_{K\ge 0}(U\times S^{l+\s})\otimes\Q$ is nonzero.
Here $k$ and the bound ``$l\gg k$'' are explicit, but $\s$ is not explicit.
The smallest $k\ge 1$ for which we know that the group is nonzero
is $k=7$, which occurs when 
$U$ is the total space of a nontrivial $\mathbb R^3$ bundle over $S^4$.

We do not yet know how to detect nontriviality 
of $\pi_k\, {\mathfrak M}_{K\ge 0}(V)$, $k\ge 1$.
The nonzero elements in $\pi_k\, {\mathfrak R}_{K\ge 0}(U\times S^{m})$
given by Theorem~\ref{thm-intro: main app} 
lie in the kernel of the $\pi_k$-homomorphism induced by the quotient map 
${\mathfrak R}_{K\ge 0}(U\times S^{m})\to {\mathfrak M}_{K\ge 0}(U\times S^{m})$.
\vspace{5pt}

{\bf Structure of the paper} In Section~\ref{sec: geom ingr}
we outline geometric ingredients of the proof with full details
given in Section~\ref{sec: indec}. Theorem~\ref{thm-intro: main} 
is proved in Section~\ref{sec: top ingr}. In Section~\ref{sec: main results}
we derive the results on $\ker\pi_{j}(\iota_N)$, 
and prove Theorem~\ref{thm-intro: main app}.
The proof involves various results on pseudoisotopy spaces
occupying the rest of the paper; many of these results are certainly
known to experts, but often do not appear in the literature in the
form needed for our purposes. Theorems~\ref{thm: localize} and 
Proposition~\ref{prop: euler char} are key ingredients in establishing nontriviality
of $\ker\pi_{j}(\iota_N)$.
\vspace{5pt}

{\bf Acknowledgments} 
We are thankful to Ricardo Andrade for
a sketch of Theorem~\ref{thm: pseudoisot is loop space}, to
John Klein for Remark~\ref{rmk: iota Dn}, and
to the referee for helping us find a sharper version
of Theorem~\ref{thm-intro: main} and other useful comments. 
The first two authors are grateful for NSF support: DMS-1105045 (Belegradek),
DMS-1206622 (Farrell). The third author was supported in part by a 
Discovery grant from NSERC.

\section{Geometric ingredients of Theorem~\ref{thm-intro: main}}
\label{sec: geom ingr}

Open complete manifolds of $K\ge 0$ enjoy a rich structure theory.
The soul construction of~\cite{CheGro}
takes as the input a basepoint of a complete open manifold $V$ of $K\ge 0$, 
and produces a compact totally convex submanifold $S$ without boundary, the so called {\it soul} of $g$, 
such that $V$ is diffeomorphic to the total space of the normal bundle of $S$. 
The soul need not contain the basepoint.

Different  basepoints sometimes produce different souls, 
yet any two souls can be moved to each 
other by a ambient diffeomorphism that restricts to an isometry on the souls, see~\cite{Sha}.
On the other hand, the diffeomorphism type and the ambient isotopy type of the soul may
depends on the metric, see~\cite{Bel, KPT, BKS-mod1, BKS-mod2, Ott-pairs, Ott-triv-norm}.

The soul construction involves asymptotic geometry so there is no a priori reason
to expect that the soul will depends continuously on the metric
varying in the smooth compact-open topology. We resolve this by imposing
the topological assumption that $V$ is
{\it indecomposable\,} meaning that $V$ admits a complete
metric of $K\ge 0$ such that the normal 
sphere bundle to a soul has no  section.
This occurs if the normal bundle to
a soul has nonzero Euler class, see Section~\ref{sec: indec} for other examples.
Also in Section~\ref{sec: indec} we explain that any indecomposable manifold $V$
has the following properties:
\begin{itemize}
\item[\textup{(i)}]
Any metric in ${\mathfrak R}_{K\ge 0}(V)$ has a unique soul, see~\cite{Yim-pseudosoul}. 
\vspace{2pt}
\item[\textup{(ii)}] 
If two metrics lie in the same path-component of ${\mathfrak R}_{K\ge 0}(V)$, 
then their souls are diffeomorphic, see~\cite{KPT}, 
and ambiently isotopic~\cite{BKS-mod1}.\vspace{2pt}
\item[\textup{(iii)}] 
The souls of any two metrics in ${\mathfrak R}_{K\ge 0}(V)$ have nonempty intersection.
\item[\textup{(iv)}] 
The normal sphere bundle to a soul $S$ of any
metric in ${\mathfrak R}_{K\ge 0}(V)$ has no   section. In particular,
$\dim(V)\le 2\dim (S)$. 
\end{itemize}

If $Q$ is a compact smooth submanifold of $V$,
we let $\mathrm{Emb}(Q,V)$ denote the space of all smooth embeddings 
of $Q$ into $V$\!. 
By the isotopy extension theorem 
the $\Diff(V)$-action on $\mathrm{Emb}(Q,V)$ by postcomposition
is transitive on each path-component, and its orbit map is a fiber bundle, see~\cite{Pal-iso-ext, Cer-iso-ext}.
The fiber over the inclusion $Q\hookrightarrow V$ is 
$\Diff(V, \mathrm{rel}\, Q)$, the subgroup of the diffeomorphisms that fix $Q$ pointwise.

The group $\Diff Q$ acts freely on $\mathrm{Emb}(Q,V)$ by
precomposing with diffeomorphisms of $Q$. 
Let $\mathcal X(Q, V)=\mathrm{Emb}(Q,V)/\Diff Q$
with the quotient topology; the orbit map is a locally trivial principal bundle, 
see~\cite{GV}.
Let $\mathcal X(V)=\coprod_Q \mathcal X(Q,V)$,
the space of compact submanifolds of $V$ with smooth topology.
Here is the main geometric ingredient of this paper.

\begin{thm}
\label{intro-thm: soul cont}
If $V$ is indecomposable, the map ${\mathfrak R}_{K\ge 0}(V)\to\mathcal X(V)$ 
that associates to a metric its unique soul is continuous.
\end{thm}

The proof is a modification of arguments in~\cite{KPT, BKS-mod1}.
What we actually use is the following version of Theorem~\ref{intro-thm: soul cont} in which
the soul is replaced by its tubular neighborhood
whose size depends continuously on the metric. For an indecomposable $V$ 
denote by $i_g$ the normal injectivity radius of a unique soul of 
$g\in {\mathfrak R}_{K\ge 0}(V)$. 
 
\begin{cor} 
\label{cor: injrad cont} 
If $V$ is indecomposable, the map ${\mathfrak R}_{K\ge 0}(V)\to (0,\infty]$ 
that associates $i_g$ to $g$ is continuous, and  
given a continuous function $\sigma\co {\mathfrak R}_{K\ge 0}(V)\to \mathbb R$
with $0<\sigma(g)<i_g$,
the map ${\mathfrak R}_{K\ge 0}(V)\to \mathcal X(V)$ that associates
to $g$ the closed $\sigma(g)$-neighborhood of its soul
is continuous.
\end{cor}
\begin{proof} Continuity of $g\to i_g$ follows from
Theorem~\ref{intro-thm: soul cont} and Lemma~\ref{lem: norm inj rad cont} below.
The other conclusion is immediate from Theorem~\ref{intro-thm: soul cont}.
\end{proof}

\section{Continuity of souls for indecomposable manifolds}
\label{sec: indec}

Throughout this section we assume that $V$ is indecomposable. 
Let us first justify the claims (i)--(iv) of Section~\ref{sec: geom ingr}. 

If a metric $g\in {\mathfrak R}_{K\ge 0}(V)$ has 
two distinct souls, then by a result of Yim~\cite{Yim-pseudosoul}
the souls are contained in an embedded submanifold, the union
of {pseudosouls\,}, that is  
diffeomorphic to $\mathbb R^l\times S$ where $l>0$, where
any soul is of the form $\{v\}\times S$. In particular,
the normal bundle to any soul of $g$ has a nowhere zero section,
so $V$ cannot be indecomposable. This implies (i).

The claim (ii) is proved in Lemma 3.1 and Remark 3.2 of~\cite{BKS-mod1}
building on an argument in~\cite{KPT}.

To prove (iii) and (iv) consider two vector bundles $\xi$, $\eta$
with closed manifolds as bases and diffeomorphic total spaces.
The associated unit sphere bundles
$S(\xi)$, $S(\eta)$ are fiber homotopy equivalent, see~\cite[Proposition 5.1]{BKS-mod1}.
By the covering homotopy property a homotopy section of a 
fiber bundle is homotopic to a section; thus having a section is a property 
of the fiber homotopy type. Hence if $\xi$
has a nowhere zero section, then so does $\eta$. 
If the zero sections of $\xi$, $\eta$ are disjoint in their common total space,
then the zero section of $\eta$ gives rise to a homotopy
section of $S(\xi)$, and hence to a nowhere zero section of $\xi$. 
These remarks imply (iii) and (iv).

\begin{proof}[Proof of Theorem~\ref{intro-thm: soul cont}] 
Since ${\mathfrak R}_{K\ge 0}(V)$ is metrizable, it suffices to show that
if the metrics $g_j$ converge to $g$ in ${\mathfrak R}_{K\ge 0}(V)$,
then their (unique) souls converge in $\mathcal X(Q, V)$.
Let $S_j$, $S$ be souls of $g_j$, $g$, respectively.
By Lemma~\ref{lem: tot geod from 0 to infinity} below
it suffices to show that $S_j$ converges to $S$ in the $C^0$ topology.
Arguing by contradiction pass to a subsequence such that each $S_j$
lies outside some $C^0$ neighborhood of $S$.
Let $p_j$, $p$ denote the Sharafutdinov retractions
onto $S_j$, $S$ for $g_j$, $g$,
 and let $\check g_j$, $\check g$ denote the metric on 
$S_j$, $S$ induced by $g_j$, $g$, respectively. 
By~\cite[Lemma 3.1]{BKS-mod1} 
$p_j\vert_S\co S\to S_j$ is a diffeomorphism for all large $j$,
and the pullback metrics $(p_j\vert_S)^*\check g_j$ converge to 
$\check g$ in the $C^0$ topology. In particular,
the diameters of $\check g_j$ are uniformly bounded.
Note that each $S_j$ intersects $S$ else $p_j\vert S$ would give rise
to a nowhere zero section of the normal bundle to $S$. Let $U$ be a compact domain in $V$
such that the interior of $U$ contains the closure of $\cup_j S_j\cup S$. 

The embedding $p_j\vert_S\co (S, \check g)\to (V, g)$ can be written as
the composition of $\mathrm{id}\co (S,\check g)\to (S, (p_j\vert_S)^*\check g_j)$,
the isometric embedding of $(S, (p_j\vert_S)^*\check g_j)$ 
onto a convex subset of $(V, g_j)$, and $\mathrm{id}\co (V, g_j)\to (V, g)$.
Recall that $C^0$ convergence
of metrics implies Gromov-Hausdorff, and hence Lipschitz convergence.
Hence the above identity map of $S$ has bi-Lipschitz constants
approaching $1$ as $j\to\infty$. Also there are compact domains
$U_j$ in $V$ and homeomorphisms $(U_j, g_j)\to (U,g)$ 
that converge to the identity and have 
bi-Lipschitz constants approaching $1$,
and hence the same is true for $p_j\vert_S\co (S, \check g)\to (V, g)$. 

By the Arzel{\`a}-Ascoli theorem $p_j\vert_S$ subconverge to 
$p_\infty\co (S, \check g)\to (V,g)$, which 
is an isometry onto its image (equipped with the metric obtained by restricting
the distance function of $g$).
Compactness of $S$ implies that $p_\infty$ is homotopic to $p_j$ for large $j$.
Since $p$ is $1$-Lipschitz
map $(V, g)\to (S,\check g)$ that is homotopic to the identity of $V$, we conclude that
$p\circ p_\infty$ is a $1$-Lipschitz homotopy self-equivalence of $(S,\check g)$. 
Homotopy self-equivalences of closed manifolds are surjective, so
$p\circ p_\infty$ is surjective, and hence compactness of $S$ implies that $p\circ p_\infty$
is an isometry. 

Set $f=p_\infty\circ (p\,\circ p_\infty)^{-1}$. Then $f(S)=p_\infty(S)$ and
$p\,\circ f$ is the identity of $S$. 
Note that $f(S)$ and $S$ intersect, else  
$f$ would give rise to a section of the normal sphere bundle to $S$. 
Fix $x\in f(S)\cap S$. Since every $S_j$ lies outside a $C^0$ neighborhood of 
$S$, there is $y\in f(S)\setminus S$. Let $u$ be a unit vector at
$p(y)$ that is tangent to a segment from $p(y)$ to $y$. Parallel translate $u$
along a segment joining $x$ and $p(y)$. By~\cite{Per-soulconj} this vector 
field exponentiate to an embedded flat totally geodesic strip, where 
$x$ lies on one side of the strip and $y$, $p(y)$ lie on the other side.
Finally, $d(p(y), x)=d(y, x)$ contradicts $y\neq p(y)$.
\end{proof}

\begin{lem} 
\label{lem: tot geod from 0 to infinity}
Given $k\in [1,\infty]$, let $g_i$ be a sequence of complete
Riemannian metrics on a manifold $M$ that $C^k$-converge 
on compact sets to a metric $g$. 
Suppose $S_i$, $S$ are totally geodesic compact submanifolds 
of $(M, g_i)$, $(M, g)$, respectively. 
If $S_i$ converges to $S$ in $C^0$-topology,
then it converges in $C^{k-1}$-topology.
\end{lem}
\begin{proof} Fix $x\in S$ and pick $r$ such that
$\exp_g\vert_x$, $\exp_{g_i}\vert_x$ are diffeomorphisms 
on the $2r$-ball centered at the origin of $T_x M$ for all
sufficiently large $i$.
Since $S_i$, $S$ are totally geodesic, they are equal to the
images under $\exp_{g_i}$, $\exp_g$ of some linear subspaces
$L_i$, $L$ of $T_{x_i} M$, $T_x M$, respectively, where $x_i$ is near $x$. 
Since $k\ge 1$, the maps $\exp_{g_i}$, $\exp_g$ are $C^0$-close, so that
$C^0$-closeness of $S_i$, $S$ implies that
$\exp_g(L_i)$ is \mbox{$C^0$-close} to $S$ in $B_g(x,r)$. 
Thus $L_i, L$ are $C^0$-close in the $r$-disk
tangent bundle over $B(x,r)$, but then they must be 
$C^\infty$-close because $C^0$-close linear subspaces
are $C^\infty$-close. Thus  $\exp_g(L_i)$, $\exp_g(L)=S$ 
are $C^\infty$-close in $B(x,r)$. Since
$\exp_{g_i}$ is $C^{k-1}$-close to $\exp_g$, we conclude that
$S_i$ is $C^{k-1}$-close to $\exp_g(L_i)$, and hence to $S$.
\end{proof}

The following lemma generalizes the well-known fact that
the injectivity radius depends continuously on a 
point of a Riemannian manifold.

\begin{lem}
\label{lem: norm inj rad cont}
Let $g_j$ be a sequence of Riemannian metrics on a manifold $M$
that converges smoothly to a Riemannian metric $g$. Let $S_j\to S$
be a smoothly converging sequence of compact boundaryless submanifolds of $M$.
Then normal injectivity radii of $S_i$ in $(M, g_i)$ 
converge to the normal injectivity radius of $S$ in $(M,g)$.
\end{lem}
\begin{proof}[Sketch of the Proof] 
Denote by $i_{g_j}$, $i_{g}$ the normal injectivity radii of $S_j$ in $(M, g_j)$,
and $S$ in $(M,g)$, respectively.  
Recall that $i_g$ equals the
supremum of all $t$ such that $d(\g(t), S)=t$ for every unit speed geodesic 
$\g$ with $\g(0)\in S$ and $\g^\prime(0)$  orthogonal to $S$.
 
Arguing by contradiction suppose  $i_{g_j}$ does not converge to $i_g$.
By passing to a subsequence we can assume that $i_{g_j}\to I\in [0,\infty]$ and $I\ne i_g$.
A standard rescaling argument implies that $I>0$.
We shall only treat the case $I<\infty$; the case $I=\infty$ is similar.
Then there is an $\e>0$ such that either $I-\e>i_g$ or $I+\e<i_g$.

If $I-\e>i_g$, there is a 
unit speed geodesic $\g$ starting on $S$ and orthogonal to $S$ at $\g(0)$ such that $d(\g(I-\e), S)<I-\e$.
This geodesic is the limit of 
unit speed $g_j$-geodesics $\g_j$ starting on $S_j$. 
Since $i_{g_j}$ tends to $I>I-\e$ we get $d_j(\g_j(I-\e),S_j)=I-\e$ for all large $j$.
The distance functions $d_j(\cdot, S_j)$ converge to $d(\cdot, S)$, hence passing to the limit gives 
$d(\g(I-\e), S)=I-\e$. This contradiction rules out the case $I-\e>i_g$.

Assume $I+\e<i_g$. Smooth convergence of metrics implies convergence
of Jacobi fields, so the $g_j$-focal radii of $S_j$ are $>I+\e$
for large $j$. Then by the well-known dichotomy there is a $g_j$-geodesic $\gamma_j$
of length $2i_{g_j}$ starting and ending on $S_j$ which is orthogonal to $S_j$ at end points. Geodesics $\gamma_j$ subconverge to a $g$-geodesic
 of length $2I$ orthogonal to $S$ at the endpoints. Therefore $i_g\le I$ which is a contradiction.
\end{proof}

\begin{ex} 
\label{ex: indecomp}
Here are some examples of indecomposable manifolds.

(1) If the normal bundle to a soul has nonzero Euler class with $\mathbb Z$ or $\mathbb Z_2$
coefficients, then $V$ is indecomposable because Euler class is an obstruction
to the existence of a nowhere zero section.

(2)  
The simplest method to produce open complete manifolds of $K\ge 0$ is to
start with a compact connected Lie group $G$ with a bi-invariant metric, 
a closed subgroup $H\le G$, and a representation $H\to O_m$ so that 
the Riemannian submersion metric on the quotient $(G\times\mathbb R^m)/H$ 
is a complete metric of $K\ge 0$ with soul $(G\times\{0\})/H$.
Any $G$-equivariant Euclidean vector bundle over $G/H$ is isomorphic
to a bundle of this form: the representation is given by 
the $H$-action on the fiber over $eH$. 
This applies to the tangent bundle $T(G/H)$
with the Euclidean structure induced by the $G$-invariant Riemannian metric on $G/H$.
When $G/H$ is orientable we conclude that  $T(G/H)$ is indecomposable if and only if 
$G/H$ has nonzero Euler characteristic (because for orientable $\mathbb R^n$ bundles
over $n$-manifolds the Euler class is the only obstruction to the existence of a nowhere zero
section).

(3)
To be indecomposable the normal bundle to a soul
need not have a nontrivial Euler class.
For example, all $\mathbb R^3$ bundles over $S^4$, $S^5$, $S^7$ admit
a complete metric of $K\ge 0$, see~\cite{GroZil}
and so do many $\mathbb R^3$ bundles over $CP^2$~\cite{GroZil-lift}. 
Their Euler classes lie in $H^3(\,\text{base}\,;\mathbb Z)=0$, yet
their total spaces are often indecomposable:

(3a) Nontrivial rank $3$ bundles over $S^n$, $n\ge 3$ do not have
a nowhere zero section else the bundle splits as a Whitney sum of a bundle of ranks $1$ and $2$ which 
must be trivial. Thus all nontrivial rank $3$ bundles over $S^4$, $S^5$, $S^7$
have indecomposable total spaces. 

(3b) By~\cite{DolWhi} oriented isomorphism classes of rank $3$ vector bundles over $CP^2$ 
are in a bijection via $(w_2, p_1)$
with the subset of \[
H^2(CP^2;\mathbb Z_2)\times HP^4(CP^2)\cong \mathbb Z_2\times\mathbb Z\]
given by the pairs $(0, 4k)$, $(1,4l+1)$, $k,l\in \mathbb Z$,
and such a bundle has a nowhere zero section if and only if $p_1$ is a square of the integer that
reduces to $w_2$ mod $2$. It follows from~\cite[Theorem 3]{GroZil-lift} 
that the total space of such a bundle is indecomposable
with three exceptions: $k$ is odd, $k$ is a square, or $l$ is the product
of two consecutive integers.

(3c)  
According to~\cite[Corollary 3.13]{GroZil} there are $88$ oriented isomorphism classes
of $\mathbb R^4$ bundles over $S^7$ that admit complete metrics of $K\ge 0$.
If the total space of such a bundle is not indecomposable, 
then it is the Whitney sum of an $\mathbb R$ bundle
and an $\mathbb R^3$ bundle (as any $\mathbb R^2$ bundle over $S^7$ is trivial).
Since there are only $12$ oriented isomorphism classes of $\mathbb R^3$ bundles 
over $S^7$, we conclude that there are at least
$76$ oriented isomorphism classes of $\mathbb R^4$ bundles over $S^7$ 
with indecomposable total spaces. In fact there are precisely 76 such bundles
because the inclusion $SO(3)\to SO(4)$ is injective on homotopy groups.
Similarly,~\cite[Proposition 3.14]{GroZil} implies that there are 
$2$ oriented isomorphism classes
of $\mathbb R^4$ bundles over $S^5$ with indecomposable total spaces.

(4) The product of any indecomposable manifold with a closed manifold of $K\ge 0$
is indecomposable. Indeed, suppose $V$ is indecomposable with a soul $S$
and $B$ is closed. If $V\times B$ were not indecomposable,
then the normal bundle to $S\times B$ in $V\times B$ would have a nowhere zero section.
Restricting the section to a slice inclusion $S\times\{\ast\}$ 
gives a section of the normal bundle of $S$ in $V$.
\end{ex} 

\begin{rmk}
The product of indecomposable manifolds need not be indecomposable. 
Indeed, let $\xi$, $\eta$ be oriented nontrivial rank two bundle over $S^2$, $RP^2$
classified by Euler classes in $H^2(S^2;\mathbb Z)\cong\mathbb Z$
and $H^2(RP^2;\mathbb Z)\cong\mathbb Z_2$. Suppose $e(\xi)[S^2]$ is even. 
Then the Euler class of $\xi\times\eta$
equals $e(\xi\times\eta)=e(\xi)\times e(\eta)$, which vanishes 
because the cross product is bilinear. By dimension reasons 
the Euler class is the only obstruction to the existence of a nowhere zero section
of $\xi\times\eta$, so the total space of $\xi\times\eta$ is not indecomposable.
In view of $(1)$ above in order to show that $\xi$, $\eta$ have indecomposable total spaces 
it is enough to give them complete metrics of $K\ge 0$. The case of $\xi$ is well-known:
Any plane bundle over $S^2$ can be realized as $(S^3\times\mathbb R^2)/S^1$,
see (2) above, so it carries a complete metric of $K\ge 0$. 
To prove the same for $\eta$ we shall identify it with
the quotient of $S^2\times\mathbb R^2$ by the involution
$i(x,v)=(-x, -v)$ which is isometric in the product of the constant curvature metrics. 
The quotient can be thought of 
$\gamma\oplus\gamma$ where $\gamma$ is the canonical line bundle over $RP^2$, so 
its total Stiefel-Whitney class equals 
$(1+w_1(\gamma))^2=1+w_1(\gamma)^2\ne 1$. Thus $\gamma\oplus\gamma$ 
is orientable and nontrivial, and hence it
is isomorphic to $\eta$ which 
is the only orientable nontrivial plane bundle over $RP^2$.
\end{rmk}

\section{Topological restrictions on indecomposable manifolds}
\label{sec: top ingr}

In this section we prove Theorem~\ref{thm-intro: main}.
Let $V=\Int N$ be an indecomposable manifold.
Fix a homeomorphism $\rho\co (0,\infty]\to (0,1]$ with $\rho(s)<s$, e.g. 
$\rho(s):=\frac{s}{s+1}$.

Fix an arbitrary metric $h\in {\mathfrak R}_{K\ge 0}(V)$
with soul $S_h$ of
normal injectivity radius $i_h$. 
By a slight abuse of notation
we identify $N$ with the $\rho(i_h)$-neighborhood of $S_h$. 
Let $\theta_h$ be the orbit map
of a metric $h\in {\mathfrak R}_{K\ge 0}(V)$ under the pullback (left) action of
$\Diff V$ given by $\theta_h(\phi)\!:=\!\phi^{-1*}h$; we sometimes denote 
$\phi^{-1*}h$ by $h_\phi$.  
Consider the diagram
\[
{\scriptsize
\xymatrix{
&\ast\simeq\Diff(V, \mathrm{rel}\, N)\ar[r]&
\Diff V\ar[d]
\ar[r]^{\theta_h\quad}&
{\mathfrak R}_{K\ge 0}(V)
\ar[d]^\delta
\\
\Omega\mathcal X(N, V) 
\ar[r]^{\Omega f}&
\Diff N\ar[r]&
\mathrm{Emb}(N, V)\ar[r]^{q}& 
\mathcal X(N, V)\ar[r]^f& 
B\Diff N
}}
\] 
The map $q$ takes an embedding to its image. Note that $q$ is a principal bundle~\cite{GV},
and $f$ denotes its classifying map. 

The leftmost vertical arrow is given by restricting to $N$, which is
a fiber bundle due to the parametrized isotopy extension
theorem. Its fiber over the inclusion $\Diff(V, \mathrm{rel}\, N)$ is contractible
by the Alexander trick towards infinity. (The fibers over other components
of $\mathrm{Emb}(N, V)$ might not be contractible, but we will only work
in the component of the inclusion).

The map $\delta$ taking $g$ to the closed $\rho(i_g)$-neighborhood
of $S_g$ is continuous by Corollary~\ref{cor: injrad cont}. 

The above diagram commutes because the isometry $\phi\co (V, h)\to (V, h_\phi)$ 
takes the $\rho(i_h)$-neighborhood
of $S_h$ to the $\rho(i_{h_\phi})$-neighborhood of $S_{h_{\phi}}$.

Let $\pi_k(\theta_h)$ be the homomorphism induced by $\theta_h$ on the $k$th homotopy groups
based at the identity map of $V$, and similarly, let $\pi_k(q)$, $\pi_k(f)$,
$\pi_k(\Omega f)$ be the induced maps of homotopy groups based at inclusions.
With these notations the commutativity of the diagram  implies that
$\im \pi_k(q)$ is a quotient of a subgroup of $\im\pi_k(\theta_h)$.

In the bottom row of the diagram every two consecutive maps form a fibration, up to homotopy. 
This gives isomorphisms $\im \pi_k(q)\cong\ker \pi_k (f)\cong\ker \pi_{k-1}(\Omega f)$ for each $k\ge 1$.

A collar neighborhood of $\d N$ defines the inclusion
$\iota_N\co P(\d N)\to\Diff N$ 
extending a pseudoisotopy on the collar neighborhood 
by the identity outside the neighborhood. 
In Theorem~\ref{thm: pseudoisot is loop space}  below
we identify the homomorphisms
$\pi_{k-1}(\Omega f)$ and $\pi_{k-1}(\iota_N)$ for each $k\ge 2$, where $\pi_{k-1}(\iota_N)$ 
is the map induced by $\iota_N$ on the $(k-1)$th homotopy group with identity maps as the basepoints. 

In summary, $\ker\pi_{k-1}(\iota_N)$ is quotient of a subgroup of $\im\pi_k(\theta_h)$
for $k\ge 2$, which completes the proof of Theorem~\ref{thm-intro: main}.

\begin{rmk} For $k=0$ and $1$ the proof of Theorem~\ref{thm-intro: main} shows that
$\im \pi_k(q)$ is a quotient of a subgroup of $\im\pi_k(\theta_h)$.
\end{rmk}

{\bf Notation:} If $\pi_j(X)$ is abelian, we let $\pi_j^\Q(X):=\pi_j(X)\otimes\Q$ and denote 
the dimension of this rational vector space by $\dim\pi_j^\Q(X)$.

\begin{rmk}
Tensoring with the rationals immediately implies that under the assumptions of
Theorem~\ref{thm-intro: main} any subspace of
$\im\pi_k^\Q(q)$ embeds into $\im\pi_k^\Q(\theta_h)$. 
\end{rmk}

\begin{rmk}
One may hope to use Theorem~\ref{thm-intro: main} to produce
infinitely generated subgroups of $\im\pi_k(\theta_h)$.
This is somewhat of an illusion because  
$\ker\pi_{k-1}(\iota_N)$ is a finitely generated abelian group 
if $\pi_1(\d N)$ is finite and $\max\{2k+7, 3k+4\}<\dim N$.
Indeed, 
$\ker\pi_{k-1}(\iota_N)$ can be identified with a subgroup of $\pi_{k+1}A(\d N)$,
see (\ref{form: waldhausen}) below,
which is finitely generated~\cite{Dwy, Ben}.
Note that all known computations of $\ker\pi_{k-1}(\iota_N)$
are in the above stability range. 
\end{rmk}

\begin{rmk}
An integral cohomology class is called 
{\it spherical\,} if it does not vanish on the image of the Hurewicz homomorphism.
In many of our examples of indecomposable $V$ the normal bundle to the soul has spherical Euler class,
which forces the soul of any metric in ${\mathfrak R}_{K\ge 0}(V)$
to have infinite normal injectivity radius by a result of 
Guijarro-Schick-Walschap~\cite{GSW}. For such $V$ the proof of Theorem~\ref{thm-intro: main}
simplifies: we need not consider $\rho$ or $i_g$, and instead
can let $\delta(g)$ be the $1$-neighborhood of $S_g$ and identify $\delta(h)$
with $N$.
\end{rmk}

\begin{rmk}
\label{rmk: iota Dn}
The map $\i_{_{D^n}}$ is injective for all homotopy groups.
Indeed, by Theorem~\ref{thm: pseudoisot is loop space} the map $f$ 
in the diagram below is a delooping of $\i_{_{D^n}}$
provided both maps are restricted to the identity components.
The leftmost horisontal arrow is given by precomposing with the inclusion,
the downward arrow is the inclusion, and the slanted arrow is their composition
\begin{equation*} 
\xymatrix{
O(n)\ar[rd]\ar[d]
&&&
\\
\Diff D^n\ar[r] & \mathrm{Emb}(D^n, \mathbb R^n)\ar[r]^{\ q} 
&\mathcal X(D^n,\mathbb R^n)\ar[r]^f & B\Diff D^n
}
\end{equation*}
The slanted arrow is a homotopy equivalence:
deform an embedding $e$ so that it fixes $0$
via $t\to te(x)+(1-t)e(0)$, then deform it to the 
its differential at $0$ via
$s\to\frac{e(sx)}{s}$, 
and finally apply a deformation retraction $GL(n,\mathbb R)\to O(n)$. 
Hence the left bottom arrow has a section which makes $q$ trivial on the
homotopy groups, so by exactness $f$ is injective on homotopy groups.
\end{rmk}

\section{Pseudoisotopy spaces, stability, and involution}
\label{sec: pi}

A {\it pseudoisotopy} of a compact smooth manifold $M$ is a diffeomorphism
of $M\times I$ that is the identity on a neighborhood of $M\times \{0\}\cup \d M\times I$.
Pseudoisotopies of $M$ form a topological group $P(M)$.
Let $P_{\mbox{\scriptsize$\hspace{-.4pt}\partial\hspace{-.4pt}$}}(M)$ denote the topological subgroup 
of $P(M)$ consisting of diffeomorphisms
of $M\times I$ that are the identity on a neighborhood of $\d (M\times I)$.

Igusa in~\cite{Igu-stability} discussed a number of inequivalent definitions
of pseudoisotopy, e.g. a pseudoisotopy is often defined as
a diffeomorphism of $M\times I$ that restricts to the 
identity of $M\times \{0\}\cup \d M\times I$. Igusa in~\cite[Chapter 1, Proposition 1.3]{Igu-stability} 
establishes a weak homotopy equivalence of pseudoisotopy spaces arising 
from various definitions, and in particular, the inclusion
\begin{equation*}
\label{form: pi inclusion diff}
P(M)\to\Diff(M\times I, \mathrm{rel}\, M\times \{0\}\cup \d M\times I)
\end{equation*}
is a weak homotopy equivalence.
The co-domain of the inclusion 
is homotopy equivalent to a CW complex;
in fact for any compact manifold $L$ with boundary and any closed subset
$X$ of $L$, the space $\Diff(L,\mathrm{rel}\, X)$ is a Fr{\'e}chet manifold~\cite[Lemma 4.2(ii)]{Yag}
and hence is homotopy equivalent to a CW complex~\cite[Lemma 2.1]{Yag}.
By contrast, we do not know if $P(M)$ is homotopy equivalent to a CW complex
which necessitates some awkward arguments in Section~\ref{sec: pseudo and sp of submfl}.

Defining a pseudoisotopy as an element of $P(M)$ is convenient
for our purposes because it allows for easy gluing: {\it 
A codimension zero embedding  of closed manifolds $M_0\to M$
induces a continuous homomorphism $P(M_0)\to P(M)$ given by extending
a diffeomorphism by the identity on $(M\setminus M_0)\times I$.}
Similarly, the map $\iota_N$ defined in the introduction is a continuous homomorphism.

By Igusa's stability theorem~\cite{Igu-stability} the stabilization map 
\begin{equation}
\label{form: stab map}
\Sigma\co P(M)\to P(M\times I)
\end{equation}
is $k$-connected if $\dim M\ge \max\{2k+7, 3k+4\}$. 
Thus the iterated stabilization 
is eventually a $\pi_i$-isomorphism for any given $i$.
The {\it stable pseudoisotopy space} $\mathscr P(M)$ is the
direct limit $\displaystyle{\lim_{m\to\infty} P(M\times I^m)}$.

It is known, see the proof of~\cite[Proposition 1.3]{Hat-surv},  
that $\mathscr P(-)$ is a functor from the category of compact manifolds and 
continuous maps to the category of topological spaces and homotopy classes of continuous maps.
Also homotopic maps $M\to M^\prime$ 
induce the same homotopy classes $\mathscr P(M)\to \mathscr P(M^\prime)$. Every $k$-connected map
$M\to M^\prime$ induces a $(k-2)$-connected map 
$\mathscr  P(M)\to \mathscr P(M^\prime)$~\cite[Theorem 3.5]{Igu-postn-pi}.

The space $P(M)$ has an involution given by $f\to \bar f$, where
\[
\bar f(x,t)=r(f(f^{-1}(x,1), 1-t))
\quad\text{and}\quad r(x,t)=(x,1-t),
\] see~\cite[p.296]{Vog}.
We write the induced involution of $\pi_i P(M)$ as $x\to \bar x$.

Since $P(M)$ is a topological group, 
the sum of two elements in $\pi_i P(M)$ is represented
by the pointwise product of the representatives of the 
elements~\cite[Corollary 1.6.10]{Spa-book}.
Hence the endomorphism of $\pi_i P(M)$ induced by the
map $f\to f\circ \bar f$ is given by $x\to x+\bar x$.

Note that the image of the map $f\to f\circ \bar f$
lies in $P_{\mbox{\scriptsize$\hspace{-.4pt}\partial\hspace{-.4pt}$}}(M)$. It follows that any element $x+\bar x\in \pi_i P(M)$
is in the image of the inclusion induced homomorphism
$\pi_i P_{\mbox{\scriptsize$\hspace{-.4pt}\partial\hspace{-.4pt}$}}(M)\to\pi_i P(M)$ for if $f$ represents $x$, then $x+\bar x$ is represented by 
$f\circ \bar f$.
For future use we record the following lemma.

\begin{lem}
\label{lem: involution}
Let $M$ be a compact manifold with boundary,  
let $i$ be an integer with $\dim (M)\ge \max\{2i+9, 3i+7\}$, and
let $\eta_i^m$ be the endomorphism of $\pi_i^\Q P(M\times I^m)$ 
induced by the map $f\to f\circ\bar f$.\vspace{-2pt}
\begin{enumerate}
\item[\textup{(1)}]
If $x\in\pi_i P(M)$ has infinite 
order, then  
$x+\bar x\in \pi_i P_{\mbox{\scriptsize$\hspace{-.4pt}\partial\hspace{-.4pt}$}}(M)$ and \vspace{2pt} 
\newline $\Sigma x+\overline{\Sigma x}\in\pi_i P_{\mbox{\scriptsize$\hspace{-.4pt}\partial\hspace{-.4pt}$}}(M\times I)$ 
cannot both have finite order.
\vspace{2pt}
\item[\textup{(2)}]
$\pi_i^\Q{\mathscr P}(M)$ embeds into $\im \eta_i^m\oplus \im \eta_i^{m+1}$.
In particular, there is $\e\in\{0,1\}$ such that 
$2\dim \im \eta_i^{m+\e}\ge \dim \pi_i^\Q{\mathscr P}(M)$.
\end{enumerate}
\end{lem}
\begin{proof}
(1) The map $f\to \bar f$ homotopy anti-commutes 
with the stabilization map (\ref{form: stab map}), as proved in~\cite[Appendix I]{Hat-surv}. 
By assumption $i$ is below Igusa's stability range so $\Sigma$ is a $\pi_i$-isomorphism, and
$\pi_i P(M)$ contains an infinite order element $x$.
Then either $x+\bar x$ or $\Sigma x+\overline{\Sigma x}$
has infinite order for otherwise
\[
2\Sigma x=\Sigma x+\overline{\Sigma x}+\Sigma x-\overline{\Sigma x}=
%\Sigma x+\overline{\Sigma x}+\Sigma x+\Sigma{\bar x}=
\Sigma x+\overline{\Sigma x}+\Sigma (x+{\bar x})
\]
would have finite order, contradicting $\pi_i$-injectivity of $\Sigma$.

(2) 
Let $\Sigma\ker \eta_i^m$ denote the image of $\ker \eta_i^m$ under 
the $\pi_i^\Q$-isomorphism induced by $\Sigma$.
The intersection of $\ker  \eta_i^{m+1}$ and $\Sigma\ker \eta_i^m$ 
is trivial, for if $x=-\bar x$ and $\Sigma x=-\overline{\Sigma x}$, 
then $\Sigma x=-\Sigma \bar x=\overline{\Sigma x}$
so that $\Sigma x=0$. Thus $\ker \eta_i^m$ injects into $\im  \eta_i^{m+1}$,
and the claim follows by observing that 
$\pi_i^\Q{\mathscr P}(M)\cong\ker \eta_i^m\oplus\im \eta_i^m$.
\end{proof}

\section{Pseudoisotopies and the space of submanifolds}
\label{sec: pseudo and sp of submfl}

Let $\Diff_0 M$, $P_0(M)$ denote the identity path-components of $\Diff M$, $P(M)$, respectively.
Given a submanifold $X$ of $Y$ let $\mathrm{Emb}_0⁡(X,Y)$ denote the 
component of the inclusion in the space of embeddings of $X\to Y$, and let
$\Omega_0\mathcal X(X, Y)$
be the component of the constant loop based at the inclusion.

If $f\co E\to B$ is a continuous map and $E_f\to B$ is the corresponding
standard fibration with a fiber $F$, then {\it the associated homotopy fiber map\,}
$F\to E$ is the composition of the inclusion $F\to E_f$ with 
the standard homotopy equivalence $E_f\to E$.

\begin{thm}\label{thm: pseudoisot is loop space}
Let $M$ be a compact manifold with nonempty boundary.
Suppose $U$ is obtained by attaching $\partial M\times [0,1)$
to $M$ via the identity map of the boundary. 
Let $l\co \Omega\mathcal X(M, U)\to \Diff M$ be the homotopy fiber map
associated with the map $\Diff M\to \mathrm{Emb}⁡(M,U)$ given by postcomposing
diffeomorphisms with the inclusion.
Then there is a weak homotopy equivalence $\phi\co \Omega_0\mathcal X(M, U)\to P_0(\partial M)$
such that $\iota_M\circ\phi$ is homotopic to the restriction of $l$ to $\Omega_0\mathcal X(M, U)$.
\end{thm}
\begin{proof}
Let $M_0$ be the complement of an open collar of $\d M$ in $M$. 
Consider the following commutative diagram:
\[
\xymatrix{
\Diff_0(M) \ar[r]^i
\ar[d]_r &  \mathrm{Emb}_0(M,U)\ar[d]_s  \\
\mathrm{Emb_0}(M_0,\Int  M)\ar[r]_j& \mathrm{Emb}_0(M_0,U).
}
\]
Here $r$ and $s$ are given by restriction to $M_0$, while $i$ and $j$
is induced by precomposing with the inclusion $M\hookrightarrow U$, and postcomposing with the inclusion 
$\Int \, M \hookrightarrow U$. 

First we show that $s$ is a homotopy equivalence.
Let us factor the restriction $\Diff_0 ⁡U\to \mathrm{Emb}⁡_0(M_0,U)$ as
the restriction $\Diff_0 ⁡U\to \mathrm{Emb}⁡_0(M,U)$
followed by $s$.
By the parametrized isotopy extension theorem~\cite{Pal-iso-ext, Cer-iso-ext} 
the above restrictions are fiber bundles with fibers $\Diff_0(U, \mathrm{rel}\, M_0)$,
$\Diff_0(U, \mathrm{rel}\, M)$, respectively. The fibers 
are contractible by the Alexander trick towards infinity, 
so $s$ is a homotopy equivalence. 

The map $j$ is also a homotopy equivalence. Note that
the space of smooth embeddings of a compact manifold into an open manifold
is an ANR because it is an open subset of a Fr{\'e}chet manifold of all smooth maps
between the manifolds. Hence the domain and codomain of $j$
are homotopy equivalent to CW complexes and
it suffices to show that $j$ is a weak  homotopy equivalence. 
This easily follows from the existence of an isotopy of $U$ that pushes a given compact subset
into $\Int M$, e.g. given a map $S^k\to \mathrm{Emb}_0(M_0,U)$
based at the inclusion we can use the isotopy to push 
the adjoint $S^k\times M_0\to U$ of the above map 
into $\Int M$ relative to the inclusion, so $j$
is $\pi_k$-surjective, and injectivity is proved similarly.

By the parametrized isotopy extension theorem 
the map $r$ is a fiber bundle, 
and its fiber $F_r$ over the inclusion 
equals the space of diffeomorphisms of $M\setminus\Int (M_0)$
that restrict to the identity of $\d M_0$ and lie in $\Diff_0 M$.
The inclusion \begin{equation}
\label{form: incl is whe}
P(\partial M)\cap \Diff_0 M\to F_r
\end{equation} 
is a weak homotopy equivalence~\cite[Chapter 1, Proposition 1.3]{Igu-stability}.
The space $F_r$ is a Fr{\'e}chet manifold,
see~\cite[Lemma 4.2(ii)]{Yag}, hence it is an ANR. 
Therefore the CW-approximation theorem
gives a weak homotopy equivalence \[
h_r\co F_r\to P(\partial M)\cap \Diff_0 M
\]
whose composition with the inclusion (\ref{form: incl is whe})
is homotopic to the identity of $F_r$.

Since $s$ and $j$ are homotopy equivalences, 
the homotopy fibers $F_i$, $F_r$ of $i$, $r$ are homotopy equivalent, i.e. 
there is a homotopy equivalence $h\co F_i\to F_r$ 
which together with the homotopy fiber maps $f_i\co F_i\to\Diff_0 M$, $f_r\co F\to\Diff_0 M$
forms a homotopy commutative triangle. 
This gives homotopies $\iota_M\circ h_r\circ h\sim f_r\circ h\sim f_i$.

Look at the map of fibration sequences
\[
{\footnotesize\xymatrix{
\Omega\Diff_0 ⁡M \ar[r]\ar[d]& \Omega\mathrm{Emb}⁡_0(M,U)\ar[r]\ar[d]&
F_i\ar[r]_{\!\!\!\!f_i}\ar[d]_g &\Diff_0 ⁡M \ar[r]\ar[d]& \mathrm{Emb}⁡_0(M,U)\ar[d]\\
\Omega\Diff M \ar[r]& \Omega\mathrm{Emb}⁡(M,U)\ar[r]&
\Omega\mathcal X(M, U)\ar[r]^l &\Diff ⁡M \ar[r]& \mathrm{Emb}(M,U)
}}
\]
where the maps in the rightmost and the leftmost squares are inclusions, 
and $g$ is the associated map of homotopy fibers.
The two rightmost vertical arrows are inclusions of path-components.
Hence the unlabeled vertical arrows induce $\pi_k$-isomorphisms
for $k>0$, and so does $g$ by the five lemma.
The space $\mathcal X(M, U)$ is a Fr{\'e}chet manifold~\cite{GV}, and hence
its loop space is homotopy equivalent to a CW complex~\cite{Mil-cw}.
Thus the restriction of $g$ to the identity component is a homotopy
equivalence whose homotopy inverse we denote by $g^\prime$.
For the map $\phi:=h_r\circ h\circ g^\prime$ we have homotopies
$\iota_M\circ\phi\sim f_i\circ g^\prime\sim 
l\circ g\circ g^\prime\sim l\vert_{\Omega_0\mathcal X(M, U)}$ as claimed.                                                                                                                                                                                                                                                                                  \end{proof}

\begin{rmk} We do not know if the groups 
$\pi_0 P(\d M)$, $\pi_0\Omega\mathcal X(M, U)$ are isomorphic.
Theorem~\ref{thm: pseudoisot is loop space} implies that 
any two path-components of  $P(\d M)$, $\Omega\mathcal X(M, U)$
are weakly homotopy equivalent. 
(If $X$ is an $H$-space whose
$H$-multiplication induces a group structure on $\pi_0(X)$, then 
all path-components of $X$ are homotopy equivalent. This applies
to topological groups and loop spaces.)
\end{rmk}

\section{Rational homotopy of the pseudoisotopy space}
\label{sec: rat hom pi}

In this section we review how to compute $\pi_*^\Q\mathscr P(M)$, 
work out the cases when $M$ is $S^n$, $HP^d$, $S^4\times S^4$, $S^4\times S^7$, and explain 
that every \mbox{$2$-connected} rational homotopy equivalence
induces an isomorphism on $\pi_*^\Q\mathscr P(-)$.

It turns out that if $M$ is simply-connected,
the computation of $\pi_*^\Q\mathscr P(M)$ reduces to a problem in the rational homotopy theory.

There is a fundamental relationship between $\mathscr P(M)$ and
the Waldhausen algebraic $K$-theory $A(M)$.
For our purposes a definition of $A(-)$ is not important, 
and it is enough to know that 
$A(-)$ is a functor 
from the category of continuous maps of
topological spaces into itself, see~\cite{Wal-PSPM78}. Let $A_f\co A(X)\to A(Y)$ 
denote a map induced by a map $f\co X\to Y$.
For each $i\ge 0$ there is a natural isomorphism
\begin{equation}
\label{form: waldhausen}
\pi_{i+2}A(M)\cong \pi_{i+2}^S(M_+)\oplus\pi_i\mathscr P(M).
\end{equation}
\vspace{-10pt}

This result was envisioned in works of Hatcher and Waldhausen
in 1970s, and a complete proof has finally appeared in~\cite[Theorem 0.3]{WJR}, 
where the notations are
somewhat different, see~\cite[section 1.15]{Rog-lec} and~\cite[p.227]{HsiSta}
for relevant background.

Here $\pi_{i+2}^S(M_+)$ is the $(i+2)$th stable homotopy group of the disjoint union
of $M$ and a point, which after tensoring with the rationals becomes
naturally isomorphic to the homology of $M$, i.e. 
$\pi_{i+2}^S(M_+)\otimes\mathbb Q\cong H_{i+2}(M;\mathbb Q)$,
see e.g.~\cite[section 20.9]{tDi}.

Dwyer~\cite{Dwy} showed that if $X$ is simply-connected and each 
$\pi_i(X)$ is finitely generated, then each $\pi_i A(X)$ is finitely generated.
Since compact simply-connected manifolds have finitely generated
homotopy groups~\cite[Corollary 9.6.16]{Spa-book}, it follows from
(\ref{form: waldhausen}) that $\mathscr P(M)$ have finitely generated homotopy
groups for each compact simply-connected manifold $M$.

The constant map $M\to \ast$ induces retractions $\mathscr P(M)\to\mathscr P(*)$ and
$A(M)\to A(\ast)$, which
give isomorphisms:
\vspace{3pt}
{\scriptsize\begin{equation}
\label{form: A and P reduced}
\pi_i\mathscr P(M)\cong\pi_i\mathscr P(*)\oplus\pi_i(\mathscr P(M), \mathscr P(*))
\qquad 
\pi_i A(M)\cong\pi_i A(\ast)\oplus \pi_{i} (A(M), A(\ast )).
\end{equation}}
\vspace{-10pt}

Waldhausen computed the rational homotopy groups of $A(\ast)$, the
algebraic $K$-theory of a point~\cite[p.48]{Wal-PSPM78}, which gives
\small\begin{equation}
\label{form: P(pt)}
\pi_q^\Q\mathscr P(*)\cong \pi_{q+2}^\Q A(\ast)= \left\{
  \begin{array}{l l}
    \mathbb Q & \quad \text{if $q\equiv 3\,(\,\mathrm{mod}\,4\,)$}\\
    0 & \quad \text{else}
  \end{array} \right. 
\end{equation}\normalsize
\vspace{-5pt}

Thus the Poincar\'e series of $\pi_*^\Q\mathscr P(*)$ is
$t^3(1-t^4)^{-1}$. Recall that the {\it Poincar\'e series} of a 
graded vector space $\underset{i}{\oplus} W_i$
is {\small $\displaystyle{\sum_i t^i\dim W_i}$}.

The Poincar\'e series of $\pi_*^\Q(\mathscr P(M), \mathscr P(*))$,
where $M=S^k$ with $k>1$, was computed in~\cite{HsiSta} as 
\begin{equation}
\label{form: P(even sphere)}
\frac{t^{3n-4}}{1-t^{2n-2}}\quad\text{if $M=S^n$ where $n\ge 2$ is even},
\end{equation}
\begin{equation}
\label{form: P(odd sphere)}
\frac{t^{4n-5}}{1-t^{2n-2}}\quad\text{if $M=S^{2n-1}$ where $n\ge 2$ is an integer}.
\end{equation}
More precisely,~\cite[pp 227-229]{HsiSta} gives the Poincar\'e series of $\pi_*A(S^k)$
and (\ref{form: P(even sphere)})-(\ref{form: P(odd sphere)}) 
is obtained from the series by subtracting the Poincar\'e series for $H_*(S^k;\mathbb Q)$ 
and $\pi_*^\Q A(\ast)$, and shifting dimensions by two.

The range of spaces $X$ for which $\pi_*^\Q A(X)$  is readily computable
was greatly extended after the discovery of a 
connection between $\pi_*^\Q A(X)$  
and $HC_*(X;\mathbb Q)$, the rational cyclic homology,
see~\cite{Goo}, and references therein. 

By~\cite[Theorem V.1.1]{Goo-cyc-hom} or~\cite[Theorem A]{BurFie} 
there is a natural isomorphism between $HC_*(X;\mathbb Q)$ and the {\it equivariant
rational homology\,} $H^{S^1}_*(LX;\mathbb Q)$. The latter
is defined as $H_*(LX\times_{S^1} ES^1;\mathbb Q)$, where
$LX\times_{S^1} ES^1$ is the Borel construction and
$LX$ is free loop space of $X$, i.e. the space
of continuous maps $S^1\to X$ with the compact-open topology.
Note that $LX$ comes with the circle action by pre-composition, and
the post-composition with a
continuous map $f\co X\to Y$ induces the $S^1$-equivariant continuous map 
$L_f\co LX\to LY$. 

The free loop space of a point is a point, so 
$H^{S^1}_*(\ast;\mathbb Q)=H_*(BS^1)$.
The map $X\to \ast$ induces a retraction $LX\times_{S^1} ES^1\to\ast\times_{S^1} ES^1=BS^1$, which
gives an isomorphism:
\vspace{3pt}
{\small
\begin{equation}
\label{form: equiv reduced}
H_i^{S^1}(LX;\mathbb Q)\cong H_i^{S^1}(\ast;\mathbb Q)\oplus
H_{i}^{S^1}(LX, \ast;\mathbb Q)
\end{equation}}
\vspace{-10pt}

In many cases $H^{S^1}_*(LX;\mathbb Q)$ can be computed due to
\begin{itemize}
\item the K\"unneth formula for rational cyclic homology $HC_*(X;\mathbb Q)$ of~\cite{BurFie};
\item a Sullivan minimal model for $LX\times_{S^1} ES^1$ developed in~\cite{VPBur} 
for any simply-connected $X$ such that $\dim \pi_i^\Q(X)$ is finite for every $i$.
\end{itemize}

To state a result in~\cite{Goo} we need a notation. 
Given a  functor $F$
that associates to a continuous map $g\co X\to Y$
a sequence of linear maps of rational vector spaces 
$g_i\co F_i(X)\to F_i(Y)$ indexed by $i\in\mathbb N$, we 
let $F_i(g)$ denote a rational vector space that fits into an exact sequence
\begin{equation}
\label{form: long exact seq}
\dots \longrightarrow F_i(X)\overset{g_i}{\longrightarrow} F_i(Y)\longrightarrow 
F_i(g)\longrightarrow F_{i-1}(X)\overset{g_{i-1}}{\longrightarrow}\dots
\end{equation}
so that $F_i(g)$ is isomorphic to direct sum of 
$\mathrm{ker}\,g_{i-i}$ and $F_i(Y)/\mathrm{Im}\,g_i$.
We apply the above when $F_i$ is the rational homotopy $\pi_i^\Q(-)$ or 
equivariant rational homology $H_{i}^{S^1}(-;\mathbb Q)$, while
$g$ is $A_f$ or $L_f$, respectively. In particular, in these notations 
$\pi_i^\Q(g)=0$
for all $i\le k$ if and only if $g$ is rationally $k$-connected.

Goodwillie proved in~\cite[p.349]{Goo}
that any $2$-connected continuous map $f\co X\to Y$ gives rise to an isomorphism
for all $i$
\begin{equation}
\label{form: goodwillie's iso}
\pi_i^\Q(A_f)\cong H_{i-1}^{S^1}(L_f;\mathbb Q).
\end{equation}

Waldhausen proved that if $f$ is $k$-connected with $k\ge 2$, then so is $A_f$, 
see~\cite[Proposition 2.3]{Wal-PSPM78}, and (\ref{form: goodwillie's iso})
gives a rational version of this result:

\begin{cor}
\label{cor: A and rational he}
If $f$ is $2$-connected and rationally $k$-connected,
then so is $A_f$.
\end{cor}
\begin{proof} 
If $F$ is a homotopy fiber of $f$, then $LF$ is a homotopy fiber of $L_f$, 
see~\cite[Theorem 5.125]{Str-book}. It is easy  to see
that $LF$ is also the homotopy fiber of the map
$LX\times_{S^1} ES^1\to LY\times_{S^1} ES^1$ induced by $L_f$.
By assumption $F$ is rationally $(k-1)$-connected, so the homotopy
exact sequence of the evaluation fibration $\Omega F\to LF\to F$ shows that
$LF$ is rationally $(k-2)$-connected. This implies that $H_{i-1}^{S^1}(L_f;\mathbb Q)=0$ for $i\le k$,
which proves the lemma thanks to (\ref{form: goodwillie's iso}).
\end{proof}

\begin{cor}
\label{cor: pi and rational he}
Any $2$-connected rationally $k$-connected map of simply-connected compact
manifolds induces an isomorphism on $\pi_i^\Q\mathscr P(-)$
for $i< k-2$ and an epimorphism for $i=k-2$.
\end{cor}
\begin{proof}
This follows from naturality of (\ref{form: waldhausen}) combined with
Corollary~\ref{cor: A and rational he} and the Whitehead theorem mod the Serre class
of abelian torsion groups~\cite[Theorem 9.6.22]{Spa-book}.
\end{proof}

If $X$ is simply-connected, then $X\to\ast$
is $2$-connected, so that (\ref{form: goodwillie's iso}) implies:
\begin{cor}
\label{cor: Goodwillie rel}
If $X$ is simply-connected, then $\pi_{i}^\Q (A(X), A(\ast ))$
is isomorphic to $H_{i-1}^{S^1}(LX, \ast;\mathbb Q)$ for all $i$.
\end{cor}
\begin{proof}
If $f\co X\to\ast$, then $A_f$, $L_f$ are retractions, so (\ref{form: long exact seq})
splits into short exact sequences. In view of (\ref{form: A and P reduced}) and 
(\ref{form: equiv reduced}), we get isomorphisms
\[
H_{i-1}^{S^1}(LX, \ast;\mathbb Q)\cong
H_i^{S^1}(L_f;\mathbb Q)\cong
\pi_{i+1}^\Q(A_f)\cong\pi_i^\Q (A(X), A(\ast ))
\]
where the middle isomorphism is given by (\ref{form: goodwillie's iso}).
\end{proof}

By (\ref{form: waldhausen}) and Corollary~\ref{cor: Goodwillie rel}
the Poincar\'e series of $\pi_*^\Q(\mathscr P(M), \mathscr P(*))$ 
equals the difference of the Poincar\'e series of $HC_{*+1}(M, \ast;\mathbb Q)$ and
$H_{*+2}(M;\mathbb Q)$. For future use we record some explicit computations of $HC_*(M)$. 

If $M$ is simply-connected and $H^*(M;\Q)\cong \Q[\a]/(\a^{n+1})$
the Poincar\'e series for $HC_*(M;\mathbb Q)$ was
found in~\cite[Theorem B]{VPBur} giving the following
Poincar\'e series for $\pi_*^\Q(\mathscr P(M), \mathscr P(*))$: 
{\small \begin{equation}
\label{form: P(HPd)}
\frac{(1-t^{4n})\,{ t^{4n+4}}}{(1-t^4)\,(1-t^{4n+2})}\qquad
\text{if $\a \in H^4(M;\Q)$}.
\end{equation}
\begin{equation}
\label{form: P(CPd)}
\frac{t^{2n}}{1-t^2}\qquad
\text{if $\a \in H^2(M;\Q)$}
\end{equation}
}
\vspace{-7pt}

In particular, (\ref{form: P(HPd)}) applies to $M=HP^n$, and (\ref{form: P(CPd)})
applies when $M$ is $CP^n$ or the total space of any nontrivial $S^2$-bundle over $S^4$,
see~\cite[Corollary 3.9]{GroZil},
which in fact is rationally homotopy equivalent to $CP^3$.

Next we compute $\pi_*^\Q(\mathscr P(M), \mathscr P(*))$
when $M$ is $S^4\times S^4$ and $S^7\times S^4$. The Poincar\'e series of
$HC_*(S^4,\ast;\mathbb Q)$ equals $t^3(1-t^6)^{-1}$~\cite[Theorem B]{VPBur},
so by dimension reasons $HC_*(S^4;\mathbb Q)$ is quasifree in the sense of~\cite[p.303]{BurFie}. 
Hence the K\"unneth formula of~\cite[Theorem B(b)]{BurFie} applies and
for any connected space $X$ we have
\[
HC_*(X\times S^4, \ast;\mathbb Q)\cong 
HC_*(X, \ast;\mathbb Q)\oplus H_*(LX;\mathbb Q)\otimes HC_*(S^4,\ast;\mathbb Q).
\] 
Recall that taking the Poincar\'e series converts $\oplus$ to the sum and $\otimes$ to the 
product of series.

Set $X=S^4$. The Poincar\'e series for $H_*(LS^4;\mathbb Q)$
is given in~\cite[Theorem B(2b)]{VPBur} and it simplifies to $1+(t^3+t^4)(1-t^6)^{-1}$.
Therefore, the Poincar\'e series for $HC_*(S^4\times S^4, \ast;\mathbb Q)$ equals
\[2t^3(1-t^6)^{-1}+(t^6+t^7)(1-t^6)^{-2},\] and we get
the Poincar\'e series for $\pi_*^\Q(\mathscr P(S^4\times S^4), \mathscr P(*))$: 
{\small \begin{equation}
\label{form: P(S4xS4)}
\frac{2t^2}{1-t^6}+
\frac{t^5+t^6}{(1-t^6)^2}-2t^2-t^6.
\end{equation}}

Set $X=S^7$. The Poincar\'e series for $HC_*(S^7, \ast;\mathbb Q)$, 
$H_*(LS^7;\mathbb Q)$ equal
$t^6(1-t^6)^{-1}$,  $(1+t^7)(1-t^6)^{-1}$, respectively. Hence
the Poincar\'e series for $HC_*(S^7\times S^4, \ast;\mathbb Q)$ equals
$t^6(1-t^6)^{-1}+t^3(1+t^7)(1-t^6)^{-2}$, and therefore, we get
the Poincar\'e series for $\pi_*^\Q(\mathscr P(S^7\times S^4), \mathscr P(*))$: 
{\small \begin{equation}
\label{form: P(S7xS4)}
\frac{t^5}{1-t^6}+
\frac{t^2(1+t^7)}{(1-t^6)^2}-t^2-t^5-t^9.
\end{equation}}

\section{Block automorphisms, pseudoisotopies, and surgery}
\label{sec: invol, surg, block}

Throughout this section $M$ is a compact manifold with (possibly empty) boundary.

Let $G(M,\partial)$ denote the space of all continuous self-maps
$(M,\partial M)$ that are homotopy equivalences of pairs that restrict to
the identity on $\partial M$, and let $\Diff (M, \d)$
be the group of diffeomorphisms that restrict to the identity of $\d M$.

Let $L_j^s(\mathbb ZG)$ denote the Wall's $L$-group of $G$ for surgery up to
simple homotopy equivalence.
These are abelian groups which are fairly well understood when $G$
is finite. In particular, if $G$ is trivial, then $L_j^s(\mathbb Z)$ 
is isomorphic to $\mathbb Z$ for $j\equiv 0 \,(\,\mathrm{mod}\,4\,)$ and is finite otherwise.

The following is known to experts but we could not locate a reference.

\begin{thm}
\label{thm: bound on pi_i Diff}
If $M$ is a compact orientable manifold and $i\ge 1$, then the dimension of 
$\pi_i^\Q\Diff (M, \d)$
is bounded above by the dimension of
\[
\Q\otimes\bigg(
\pi_i G(M, \d)\oplus\pi_i {\mathscr P}(M)\oplus
L^s_{q+1}(\mathbb Z\pi_1M)\oplus\Big(\underset{l\in\mathbb Z_+}{\oplus} H_{q-4l}(M)\Big)\bigg)\] 
provided $3i+9<\dim M$
and $q=i+1+\dim M$.
\end{thm}
\begin{proof}
Every topological monoid with the identity has abelian fundamental group
so tensoring its $i$th homotopy group with $\Q$ makes sense for $i\ge 1$.

Let $\widetilde{\mathrm{G}}(M,\d )$ be the topological monoid of block
homotopy equivalences of $(M, \d M)$ that are the identity on the 
boundary, and let $\widetilde{\Diff}(M,\d )$ be the subgroup
of block diffeomorphisms (see e.g.~\cite{BerMad} for background on block automorphisms). 
The inclusion $G(M,\d )\to \widetilde G(M,\d )$ is a homotopy equivalence, see~\cite[p.21]{BerMad}
and there is a fibration
\[
\widetilde G(M,\d )/\widetilde{\Diff}(M,\d )\to B\widetilde{\Diff}(M,\d )\to
B\widetilde{G}(M,\d )
\]
whose homotopy sequence gives
\begin{equation}
\label{form: diff-tilde and G}
\dim\pi_i^\Q\,\widetilde{\Diff}(M,\d )\le
\dim\Big(\pi_i^\Q G(M,\d )\oplus\pi_{i+1}^\Q\widetilde G(M,\d )/\widetilde{\Diff}(M,\d )\Big).
\end{equation} 
Hatcher~\cite[Chapter 2]{Hat-surv} constructed a spectral sequence $E^n_{pq}$
converging to $\pi_{p+q+1}\widetilde{\Diff}(M,\d )/\Diff(M,\d )$
with \[
E^1_{pq}=\pi_qP(M\times D^p)\quad\text{and}\quad 
E^2_{pq}=H_p(\mathbb Z_2; \pi_q\mathscr P(M))\]
for $q\ll p+\dim(M)$.
All elements in $H_{p>0}(\mathbb Z_2; -)$ have order $2$~\cite[Proposition III.10.1]{Bro-book},
so rationally only the terms $E^2_{0q}$ can be nonzero. 
Hatcher's arguments combined with Igusa's stability theorem~\cite{Igu-stability} 
show that 
$\pi_{q+1}^\Q(\widetilde{\Diff}(M,\d ),\Diff(M,\d ))$ 
is a quotient of $E^1_{0q}\otimes\mathbb Q=\pi_q^\Q {\mathscr P}(M)$  
provided $\max\{10, 3q+9\}<\dim(M)$.
Thus the homotopy exact sequence of the pair $(\widetilde{\Diff}(M,\d ),\Diff(M,\d ))$
implies for $i\ge 1$ and $3i+9<\dim(M)$:
\begin{equation}
\label{form: from hatcher spectral sequence}
\dim\pi_{i}^\Q\Diff(M,\d )\le 
\dim\Big(\pi_{i}^\Q\widetilde{\Diff}(M,\d )\oplus\pi_{i}^\Q {\mathscr P}(M)\Big). 
\end{equation}
Surgery theory allows us to
identify $\pi_{i+1}\widetilde G(M,\d )/\widetilde{\Diff}(M,\d )$ with
the relative smooth structure set $\mathcal S(M\times D^{i+1}, \d)$,
see~\cite{Qui} and~\cite[p.21-22]{BerMad}. Set $Q=M\times D^{i+1}$ and $q=\dim Q$.
If $\dim Q>5$ and $i\ge 0$, then the surgery exact sequence 
\begin{equation}
\label{form: surgery sequence}
L_{1+\dim Q}^s(\mathbb Z\pi_1 Q)\to \mathcal S(Q, \d )\to [Q/\d Q, F/O]\to
L_{\dim Q}^s(\mathbb Z\pi_1 Q)
\end{equation}
is an exact sequence of abelian groups, where 
$F/O$ is the homotopy fiber of the $J$-homomorphism $BO\to BF$. 
Since $BF$ is 
rationally contractible, the fiber inclusion $F/O\to BO$ is a rational homotopy equivalence,
hence rationally $F/O$ is the product of Eilenberg-MacLane spaces $K(\mathbb Z, 4l)$, $l\in\mathbb Z_+$.
It follows that 
\[
[Q/\d Q,F/O]\otimes\mathbb Q\cong \underset{l\in\mathbb Z_+}{\oplus} H^{4l}(Q/\d Q;\mathbb Q).
\] 
where by the Poincar\'e-Lefschetz duality
\[
\widetilde H^{j}(Q/\d Q;\mathbb Q)\cong H^{j}(Q,\d Q;\mathbb Q)\cong H_{\dim Q-j}(Q;\mathbb Q)
\cong H_{\dim Q-j}(M;\mathbb Q)
\]
which  completes the proof because of 
(\ref{form: diff-tilde and G}), (\ref{form: from hatcher spectral sequence}), 
(\ref{form: surgery sequence}).
\end{proof}

\begin{cor} 
\label{cor: pi=0 implies diff=0}
Let $M$ be a compact simply-connected manifold and let $i\ge 1$ 
such that $\pi_i^\Q G(M, \partial)=0$
and $3i+9<\dim M$. Let $q=\dim M +i +1$.
If one of the following is true 
\vspace{-4pt} 
\begin{itemize} 
\item 
$q$ equals $0$ or $1$ $\mathrm{mod}\ 4$, and 
$\tilde H_*(M;\mathbb Q)=H_{2r}(M;\Q)$ for some odd $r$, 
\item 
$q$ equals $1$ or $2$ $\mathrm{mod}\ 4$,
and $\tilde H_*(M;\mathbb Q)\cong \underset{r\in\mathbb Z_+}{\oplus} H_{4r}(M;\Q)$,
\end{itemize}
\vspace{-7pt}
then  $\dim\pi_i^\Q \Diff(M,\d )\le\dim\pi_i^\Q P(M)$.
\end{cor}
\begin{proof}
The assertion is a consequence of Theorem~\ref{thm: bound on pi_i Diff} except when $q=0\,(\mathrm{mod}\ 4)$.
But in this case we can remove $H_{0=q-4l}(M)\cong\mathbb Z$ from the right hand side of the inequality
in the statement of Theorem~\ref{thm: bound on pi_i Diff}
because in (\ref{form: surgery sequence}) 
the surgery obstruction map $[Q/\d Q, F/O]\to L_q^s(\mathbb Z)\cong\mathbb Z$ is nonzero.
We could not find this stated in the literature, so here is a proof. Recall that a normal map is
a morphism of certain stable vector bundles whose restriction to the zero sections is a degree one map
that is a diffeomorphism on the boundary.
By plumbing, see~\cite[Theorems II.1.3]{Bro-surg-book}, for every integer $n$ one can find a 
compact manifold $P$ and a degree one map 
$(P, \d P)\to (D^{q=4l}, \d D^q)$ that restricts to a homotopy equivalence $\d P\to \d D^q$,
is covered by a morphism from the stable normal bundle
of $P$ to the trivial bundle over $D^q$, and whose surgery obstruction equals $n$.
The group of homotopy $(q-1)$-spheres is finite, so 
by taking boundary connected sums of this normal map with itself
sufficiently many, say $k$, times we can arrange that the homotopy sphere 
$\d P$ is diffeomorphic to $\d D^q$; the surgery obstruction then equals $kn$. 
The map $\d P\to D^q$ preserves the orientation, so 
identifying $\d P$ with $\d D^q$ yields a self-map of $\d D^n$ that is homotopic
to the identity. Attaching the trace of this homotopy to $P$ we can assume
that $\d P\to \d D^q$ is the identity.
Let $L$ be the manifold built by replacing an embedded $q$-disk in $\Int Q$ with $P$, so that
there is a degree one map $(L, \d L)\to (Q, \d Q)$ that equals the identity outside 
the embedded copy of $P$. The bundle data match because the restriction of the stable normal bundle
of $P$ to $\d P$ is the stable normal bundle to $\d P$, which is trivial. The additivity of the surgery obstruction, see~\cite[II.1.4]{Bro-surg-book},
shows that the surgery obstruction of the above normal map covering $(L, \d L)\to (Q, \d Q)$ equals $kn$.
\end{proof}

\section{Manifolds for which $\iota_N$ is not injective on rational homotopy}
\label{sec: main results}

In this section we derive criteria of when $\iota_N$ is not injective on rational homotopy
groups and verify the criteria for manifolds in Theorem~\ref{thm-intro: main app}.
To apply results of Section~\ref{sec: invol, surg, block} we need to bound
the size of $\pi_i G(M, \partial)$.

\begin{prop}
\label{prop: bound on G}
If $E$ is a compact simply-connected manifold with $\pi_l^\Q(E)=0$ for all $l\ge n$, 
then $\pi_i G(E\times D^m, \partial)$ is finite for all $m\ge \max\{0, n-i\}$.
\end{prop}
\begin{proof}
Since $E$ is compact simply-connected, $\pi_l E$ is finitely generated
for all $l$, see~\cite[Corollary 9.6.16]{Spa-book}, so $\pi_l E$ is finite for $l\ge n$.
For any $m\ge \max\{0, n-i\}$ 
\[
\dim(E\times D^m)+i-n\ge\dim E +\max\{0, n-i\}+i-n\ge \dim E,
\] 
so $H_j(E\times D^m)=0$ for $j>\dim(E\times D^m)+i-n$
and the claim follows by applying Lemma~\ref{lem: G has finite homotopy} below to $M=E\times D^m$.
\end{proof}
 
\begin{rmk}
To apply the above proposition we either fix any $n$, $i$ and pick
$m$ large enough, or assume $i\ge n$ and let $m$ be arbitrary. 
Note that if $M$ a rationally elliptic manifold, then $\pi_i^\Q(M)=0$ for all 
$i\ge 2\sup\{l\co H_l(M;\Q)\neq 0\}$, see~\cite[Theorem 32.15]{FHT-book}.
\end{rmk}

\begin{lem} 
\label{lem: G has finite homotopy} 
Let $M$ be a compact orientable manifold such that for each $l$
the group $\pi_l M$ is finitely generated and 
$\pi_1 M$ acts trivially on  $\pi_l(M)$.
If $\pi_l(M)$ is finite for all $l\ge n$ and 
$H_j(M)$ is finite for all $j>\dim(M)+i-n$, 
then $\pi_i G(M, \partial)$ is finite.
\end{lem}
\begin{proof} 
Arguing by contradiction suppose $\pi_i G(M, \partial)$
contains an infinite sequence of elements represented by maps 
$f_k\co (D^i, \partial D^i)\to G(M, \partial)$. 
The adjoint $\hat f_k\co M\times D^i\to M$ of the map $f_k$ 
restricts to the identity of $\partial (M\times D^{i})$.
Adjusting $f_k$ within its homotopy class and 
passing if necessary to a subsequence we can find $l\ge 1$ such 
that $\hat f_k$ all agree on the $(l-1)$-skeleton
and are pairwise non-homotopic on the $l$-skeleton rel boundary. 
Denote by $1$ the map sending $(D^i, \partial D^i)$
to the identity element of $G(M, \partial)$, and let $\hat 1$ be its adjoint.

The rest of the proof draws on the obstruction theory as e.g. in~\cite{MosTan}
which applies as $\pi_1(M)$ acts trivially on homotopy groups. 
The difference cochain $d(\hat f_k, \hat 1)$ that occurs in trying to homotope
$\hat f_k$ to  $\hat 1$ over the $l$-skeleton relative to the boundary is a cocycle representing
a class in the group $H^l(M\times D^{i}, \partial (M\times D^{i}); \pi_l M)$, which
by Poincar\'e-Lefschetz duality is isomorphic to 
$H_{\dim M+i-l}(M\times D^{i}; \pi_l M)\cong H_{\dim M+i-l}(M; \pi_l M)$. 

Let us show that $H_{\dim M+i-l}(M; \pi_l M)$ is finite.
If $l\ge n$, this follows from finiteness of $\pi_l M$ and compactness of $M$. 
If $l<n$, then $H_{\dim M+i-l}(M)$ is finite by assumption as $\dim M+i-l>\dim M+i-n$.
Since $\pi_l M$ is finitely generated
for all $l$, the group $H_{\dim M+i-l}(M; \pi_l M)$
is finite by the universal coefficients theorem. 

Hence passing to a subsequence we can assume that 
$d(\hat f_k, \hat 1)$ are all cohomologous, 
which by additivity of difference cochains implies
that $d(\hat f_k, \hat f_s)$ is a coboundary for all $s,k$.
Thus all $\hat f_k$ are homotopic on the $l$-skeleton
rel boundary, which contradicts the assumptions.
\end{proof}

The following result, combined with upper bounds on the rational homotopy 
of the diffeomorphism group obtained in Section~\ref{sec: invol, surg, block},
yields a lower bound on $\dim\ker\pi_i^\Q(\iota_N)$ in terms of rational homotopy
groups of stable pseudoisotopy spaces, which in many cases 
can be computed.

\begin{thm} 
\label{thm: localize}
If $E$ is a compact manifold,  and $k$, $i$ are integers 
such that $k\ge 0$, $i\ge 1$ and $\max\{2i+9, 3i+7\}<k+\dim \d E$, then
there is $\e=\e(E,i,k)\in\{0,1\}$ such that 
\[
\dim\ker\pi_i^\Q\big(\iota_{_{E\times {S^{k+\e}}}}\big)\ge 
\frac{\dim\pi_i^\Q {\mathscr P}(\d E)}{2}-\dim\pi_i^\Q\Diff (E\times D^{k+\e}, \d).
\]
\end{thm}
\begin{proof}
Set $d_i=\dim\pi_i^\Q {\mathscr P}(\d E)$. 
Lemma~\ref{lem: involution}(ii) applied to the manifold $D^k\times \d E$
shows the existence of $\e\in \{0,1\}$
such that the image of $\pi_i^\Q$-homomorphism
induced by the inclusion \[
P_{\mbox{\scriptsize$\hspace{-.4pt}\partial
\hspace{-.4pt}$}}(D^{k+\e}\times \d E)\to P(D^{k+\e}\times \d E)\]
has dimension $\ge \frac{d_i}{2}$.

Set $m=k+\e$ and $N=E\times S^m$. Let $D^m$ denote the upper hemisphere of $S^m$, and 
set $D=E\times D^m$ with the corners smoothed. Let $\Diff^{J}\!(D, \d)$ be the subgroup of $\Diff (D, \d)$
consisting of diffeomorphisms whose $\infty$-jet at $E\times \d D^m$
equals the $\infty$-jet of the identity map.
Following~\cite[Chapter 1, Proposition 1.3]{Igu-stability} one can show that
the inclusion $\Diff^J\hspace{-1pt}(D, \d)\to \Diff (D, \d)$ is a weak homotopy equivalence.
Consider the following commutative diagram of continuous maps
\begin{equation} 
\xymatrix{
P(D^m\times \d E)\ar[rd]_{\tau} &
P_{\mbox{\scriptsize$\hspace{-.4pt}\partial\hspace{-.4pt}$}}(D^m\times \d E)
\ar[l]^{\s} \ar[r]_{\!\!\!\iota} &
\Diff^J\hspace{-1pt}(D^m\times E, \partial)\ar[d]^\r
\\
&
P(S^m\times \d E)
\ar[r]^{\iota_N} &
\Diff(S^m\times E)
}
\end{equation}
in which $\s$ is the inclusion, the maps $\tau$, $\r$ extend diffeomorphisms by the identity,
and $\iota$ is the restriction of $\iota_N$. The reason we have to deal with
$\infty$-jets is that the extension of a diffeomorphism in $\Diff(D, \d)$
by the identity of $N$ is not a diffeomorphism.

The inclusions $\d E\to D^m\times\d E\to S^m\times\d E$ induce $\pi_i^\Q$-monomorphisms
of stable pseudoisotopy spaces as $S^m\times\d E$ retracts onto $\d E\to D^m$.
The same is true unstably since $i$ is in Igusa's stable range.
Thus there is a subspace $W$ of $\pi_i^\Q P_{\mbox{\scriptsize$\hspace{-.4pt}\partial\hspace{-.4pt}$}}(D^m\times \d E)$ of dimension $\ge \frac{d_i}{2}$
that is mapped isomorphically to a subspace $U$ of $\pi_i^\Q P(S^m\times \d E)$ by $\tau\circ\sigma$.
Hence the kernel of $\pi_i^\Q(\iota)\vert_{W}$ embeds
into the kernel $\pi_i^\Q(\iota_N)\vert_U$, and the kernel of 
$\pi_i^\Q(\iota)\vert_W$ clearly satisfies
the claimed inequality.
\end{proof}

\begin{rmk} Sadly, there is not a single example of $E$, $i$, $k$ with indecomposable $\Int E$
for which we know the value of $\e$.
\end{rmk}
	
\begin{prop}
\label{prop: signly generated}
Let $E$ be the total space of a linear disk bundle over a closed manifold 
such that $E$ and $\d E$ are simply-connected,
the algebra $H^*(E;\Q)$  has a single generator, and 
the algebra $H^*(\d E;\Q)$  does not have a single generator.  Then 
there are sequences $i_l$, $m_l$ such that the sequence
 $\dim\ker\pi_{i_l}^\Q\big(\iota_{_{E\times{S^{^{m_l}}}}}\big)$ is unbounded.
\end{prop}
\begin{proof} 
By~\cite[Corollary 2]{VPBur} the sequence 
$\dim HC_i(E;\Q)$ is bounded while $\dim HC_i(\d E;\Q)$ is unbounded. 
Since $0\le \dim\pi_i^\Q {\mathscr P}(\ast)\le 1$, we conclude (see Section~\ref{sec: rat hom pi})
that the sequence $\dim\pi_i^\Q{\mathscr P}(E)$ is bounded and  $\dim\pi_i^\Q{\mathscr P}(\d E)$
is unbounded. 
The class of rationally elliptic spaces contains all
closed manifolds whose rational cohomology algebra has $\le 2$ generators, 
and is closed under fibrations, see~\cite{FHT}, so $E$ is rationally elliptic.
Hence Proposition~\ref{prop: bound on G} applies for all sufficiently large $i$  and any $m$,
and we have $\pi_i^\Q G(E\times D^m,\d )=0$, which by Theorem~\ref{thm: bound on pi_i Diff}
gives a uniform upper bound on $\dim \pi_i^\Q\Diff(E\times D^m, \d)$, and 
the result follows from Theorem~\ref{thm: localize}.
\end{proof}

\begin{rmk}
If in Proposition~\ref{prop: signly generated} 
the algebras $H^*(\d E;\Q)$, $H^*(E;\Q)$ are singly generated, 
we can still compute the dimensions of $\pi_i^\Q{\mathscr P}(\d E)$, $\pi_i^\Q{\mathscr P}(E)$
using~\cite[Theorem B]{VPBur}. In view of Section~\ref{sec: invol, surg, block}
and Theorem~\ref{thm: localize} this gives a computable lower bound on
$\dim\ker\pi_{i}^\Q\big(\iota_{_{E\times{S^{^m}}}}\big)$; of course the bound
might be zero.
\end{rmk}

Let us investigate when Proposition~\ref{prop: signly generated} does not apply.

\begin{lem}
\label{lem: sphere bundle singly generated}
Let $p\co T\to B$ be a linear $S^{k}$-bundle over a closed manifold $B$
with $\dim B>0$ 
such that $T$, $B$ are simply-connected 
and $H^*(T; \Q)$ is singly generated, and let $e$ be 
the rational Euler class of $p$. Then $k<\dim B$ and  the following holds:
\vspace{-6pt}
\begin{itemize}
\item[\textup{(1)}] If $B=S^d$, then either $e=0$ and $\frac{d}{2}=k$ is even, \newline
\phantom{\textup{(1)}\qquad\qquad\qquad\ \ \ }\hspace{1pt} 
or $e\neq 0$ and $d=k+1$ is even.
\item[\textup{(2)}] If $B=CP^d$ with $d\ge 2$, then $k=1$ and $e\neq 0$.
\item[\textup{(3)}] If $B=HP^d$ with $d\ge 2$, then either $k=2$, or $k=3$ and $e\neq 0$.
\end{itemize}
\end{lem}
\begin{proof} This is a straightforward application of the Gysin sequence 
{\tiny\begin{equation*}
\xymatrix{
\textup{\!\!\!(\bf{G})}\ \ \ H^{j-k-1}(B;\Q)\ar[r]_{\quad\cup e} & H^j(B; \Q) \ar[r]_{p^*} & H^j(T; \Q) 
\ar[r] & H^{j-k}(B; \Q)\ar[r]_{\!\cup e} & H^{j+1}(B; \Q).
}
\end{equation*}}

\vspace{-12pt}

If $\dim B\le k$, then $e=0$ for dimension reasons, so $p^*$ is injective and
$H^k(T; \Q)$ surjects onto $H^0(B; \Q)\cong \Q$.  
If $\dim B=k$, then $\dim H^k(T; \Q)=2$ contradicting that $H^*(T; \Q)$ is singly generated.
If $\dim B<k$ and $H^*(T; \Q)=\langle a\rangle$, then $a$ has degree $\le \dim B<k$ and hence $a\in\im p^*$
so that $p^*$ is a surjection
of $H^k(B; \Q)=0$ onto $H^k(T; \Q)\cong \Q$, which is a contradiction. 
Thus $k<\dim B$.

Let $B=S^d$. Then ({\bf G}) implies $H^j(T; \Q)=0$ except for $j=0\, ,\, k+d$, and possibly for 
$j=k, d$. If $e\neq 0$, then $d=k+1$ is even, and ({\bf G}) gives
$H^k(T;\Q)=0=H^d(T;\Q)$. 
If $e=0$, then ({\bf G}) shows that $H^j(T; \Q)$ are nonzero for $j=k, d$. 
Since $H^*(T; \Q)$ is singly generated, $k$, $d$ 
must be even because an odd degree class is not a power of an
even degree class, and any odd degree class has zero square. 
As $k<d$, we have $d=2k$ completing the proof of (1).

To prove (2) let $B=CP^d$ and note that simple connectedness of $T$ shows that
if $k=1$, then $e\neq 0$. To rule out $k\ge 2$ use 
({\bf G}) to conclude that
$p^*\co H^2(B)\to H^2(T)$ is injective, hence as $H^*(T;\Q)$ is singly generated,
the generator must come from $B$ and hence its $(n+1)$th power is zero,
but then it cannot generate the top dimensional class
in degree $\dim T=k+\dim B\ge 2+2n$.

To prove (3) let $B=HP^d$. Similarly to (2) if $k\ge 4$, then
$H^*(T;\Q)$ is not singly generated. The same
holds for $k=1$ as then $T\to B$ is the trivial $S^1$-bundle because $HP^d$
is $2$-connected. Thus $k$ must equal $2$ or $3$. 
Finally, if $e$ were zero for $k=3$, then 
({\bf G}) gives that $H^3(T;\Q)$ and $H^4(T;\Q)$
are nonzero, so $H^*(T;\Q)$ could not be singly generated.
\end{proof}

\begin{rmk}
(a) The exceptional cases above do happen. 
Examples are the unit tangent bundle to $S^{d}$ with $d$ even (whose total space
is a rational homology sphere),
the Hopf bundles {\footnotesize $S^1\to S^{2d+1}\to CP^d$} and 
{\footnotesize $S^3\to S^{4d+3}\to HP^d$},
and the canonical $S^1$ quotient {\footnotesize $S^2\to CP^{2d+1}\to HP^d$} of the latter bundle. 
All nontrivial $S^2$-bundles over $S^4$ have singly generated total space,
see~\cite[Corollary 3.9]{GroZil}.
Each of these total spaces
appears as $\d E$ where $\Int E$ admits a complete metric of $K\ge 0$.

(b) The assumption that $B=S^n, CP^n$ or $HP^n$
is there only to simplify notations
by excluding some cases not relevant to our geometric applications. 
The proof of Lemma~\ref{lem: sphere bundle singly generated}
applies to some other bases, e.g. the Cayley plane or biquotients
with singly generated cohomology, which are classified in~\cite{KapZil}. 
In particular, the unit tangent
bundle to the Cayley plane does not have singly generated cohomology.

(c) One can use results of~\cite{Hal-rat-fibr} to give a rational characterization of
fiber bundles $T\to B$ such that $T$, $B$ are simply-connected manifolds and $H^*(T; \Q)$
is singly generated. We will not pursue this matter because with the exception 
mentioned in (b) it is unclear if such bundles
arise in the context of nonnegative curvature.
\end{rmk}

\begin{thm}
\label{thm: bounds on P(E) and P(dE)}
Let $N=S^m\times E$ and $i\le m-3$ where $E$ and $i$ satisfy one of the following:
\vspace{-3pt}
\begin{itemize}
\item[\textup{1.}]
$E$ is  the total space of a linear $D^{2d}$-bundle over $S^{2d}$, $d\ge 2$,
with nonzero Euler class, and $i=8d-5+j(4d-2)$ for some odd $j\ge 1$. 
\vspace{2pt}
\item[\textup{2.}]
$E$ is the total space of a linear $D^4$-bundle over $HP^{d}$, $d\ge 1$, with
nonzero Euler class, 
and $i=8d+3+j(4d+2)$ for some odd $j\ge 1$. 
\vspace{2pt}
\item[\textup{3.}]
$E$ is the total space a linear $D^3$-bundle over $HP^d$, $d\ge 1$
with nonzero first Pontryagin class,
and $i=4d+2+j(2d+1)$ for some even $j\ge 0$. 
\vspace{2pt}
\item[\textup{4.}]
$E$ is the product of $S^4$ and
the total space of a $D^4$-bundle over $S^4$
with nonzero Euler class, and $i=6j+3$ for some
odd $j\ge 3$. 
\end{itemize}
\vspace{-3pt}
Then $\pi_i^\Q\mathscr P(N)=0$, and furthermore, 
$\dim\pi_i^\Q\mathscr P(\d N)=1$ in the cases 
\textup{(1)}, \textup{(2)}, \textup{(3)} and $\dim\pi_i^\Q\mathscr P(\d N)=j$
in the case \textup{(4)}.
\end{thm}
\begin{proof}
The inclusions $E\to N$ and $\d E\to \d N$ are $(m-1)$-connected,
so they induce isomorphisms on $\pi_i^\Q\mathscr P(-)$ for $i\le m-3$.

{\bf Case 1.} Here $\d E$ is the total space of
$S^{2d-1}$-bundle over $S^{2d}$, so the homotopy sequence of the bundle
shows that $\d E$ is $2$-connected while the Gysin sequence implies
$H^*(\partial E;\mathbb Q)\cong H^*(S^{4d-1};\mathbb Q)$. 
So any degree one map $\d E\to S^{4d-1}$
is a rational homology isomorphism, and hence a rational homotopy equivalence.
The map is $2$-connected, so by
Corollary~\ref{cor: pi and rational he} it induces 
an isomorphism on $\pi_i^\Q\mathscr P(-)$.
Now
(\ref{form: P(pt)}), (\ref{form: P(even sphere)}), (\ref{form: P(odd sphere)})
give the Poincar{\'e} polynomials
{\small\begin{equation*}
\frac{t^3}{1-t^4}+\frac{t^{6d-4}}{1-t^{4d-2}}\ \ \text{for $\pi_*^\Q{\mathscr P}(S^{2d})$}\quad\
\text{and}\quad\
\frac{t^{8d-5}}{1-t^{4d-2}}\ \ \text{for $\pi_*^\Q(\mathscr P(S^{4d-1}), \mathscr P(\ast))$}.
\end{equation*}}
\vspace{-13pt}

Reducing the exponents mod $4$ yields the desired conclusion.

{\bf Case 2.} Here $\d E$ is a a simply-connected rational homology $S^{4d+3}$.
Then (\ref{form: P(pt)}), (\ref{form: P(HPd)}), (\ref{form: P(odd sphere)}) give
the Poincar{\'e} polynomials
{\small\begin{equation*}
\frac{t^3}{1-t^4}+\frac{(1-t^{4d})\,{ t^{4d+4}}}{(1-t^4)\,(1-t^{4d+2})}
\ \ \text{for $\pi_*^\Q{\mathscr P}(HP^d)$}
\end{equation*}}
{\small\begin{equation*}
\frac{t^{8d+3}}{1-t^{4d+2}}\ \ \text{for $\pi_*^\Q(\mathscr P(S^{4d+3}), \mathscr P(\ast))$}.
\end{equation*}}
\vspace{-13pt}

Reducing the exponents mod $4$ implies the claim.

{\bf Case 3.}  
Nontriviality of the first Pontryagin class implies, see~\cite[pp. 273-274]{Mas-sph-bund},
that the algebra $H^*(\d E;\Q)$ is isomorphic to $\Q[\a]/\a^{2d+2}$ 
for some $\a\in H^2(\d E;\Q)$. Then 
(\ref{form: P(pt)}), (\ref{form: P(HPd)}), (\ref{form: P(CPd)}) give
the Poincar{\'e} polynomials
{\small\begin{equation}
\label{form: HPd pi}
\frac{t^3}{1-t^4}+\frac{(1-t^{4d})\,{ t^{4d+4}}}{(1-t^4)\,(1-t^{4d+2})}
\ \ \text{for $\pi_*^\Q{\mathscr P}(HP^d)$}
\end{equation}}
{\small\begin{equation}
\label{CPn pi}
\frac{t^{2(2d+1)}}{1-t^{2}}\ \ \text{for $\pi_*^\Q(\mathscr P(\d E), \mathscr P(\ast))$}.
\end{equation}}
\vspace{-13pt}

The monomials with exponent $i$  appear in (\ref{CPn pi}) and do not occur in
the first summand of (\ref{form: HPd pi}). The second summand can be written as
$\sum_{s=1}^d t^{4d+4s}\sum_{r\ge 0} t^{(4d+2)r}$,
hence the exponents of its monomials are $4d+4s+(4d+2)r$
which all lie in the union of the disjoint intervals $[4d+4+(4d+2)r,\, 8d+(4d+2)r]$.
Each number $4d+2+(4d+2)r$ lies in the gap between the intervals, so
letting $j=2r$ completes the proof. In fact, many more values of $i$ are allowed
because the exponents $4d+4s+(4d+2)r$ of distinct pairs $(s,r)$ differ by $4$ or $6$
while (\ref{CPn pi}) contains every even exponent $\ge 4d+2$.

{\bf Case 4.}  Here $\d E$ is $2$-connected rational homology $S^4\times S^7$. 
Then (\ref{form: P(pt)}), (\ref{form: P(S4xS4)}), (\ref{form: P(S7xS4)})
give the Poincar{\'e} polynomials 
{\small \begin{equation}
\label{form: P(S4xS4) bis}
\frac{t^3}{1-t^4}+\frac{2t^2}{1-t^6}+
\frac{t^5+t^6}{(1-t^6)^2}-2t^2-t^6 \ \ \ \text{for $\pi_*^\Q{\mathscr P}(S^4\times S^4)$}
\end{equation}}
{\small \begin{equation}
\label{form: P(S7xS4) bis}
\frac{t^5}{1-t^6}+
\frac{t^2+t^9}{(1-t^6)^2}-t^2-t^5-t^9
\ \ \ \text{for $\pi_*^\Q(\mathscr P(\d E), \mathscr P(\ast))$}.\end{equation}}
\vspace{-5pt}

The term {\small\[
\frac{t^9}{(1-t^6)^2}-t^9=
\displaystyle{\sum_{j\ge 2}} j\,t^{6j+3}\]} \vspace{-5pt}

in (\ref{form: P(S7xS4) bis}) has exponents
that reduce to $3$ mod $6$, so it has some common exponents only with the term
$t^3(1-t^4)^{-1}$ in (\ref{form: P(S4xS4) bis}). 
The exponents corresponding to odd $j$ reduce to $1$ mod $4$,
so do not appear in (\ref{form: P(S4xS4) bis}). For the same reasons the exponents
do not appear elsewhere in (\ref{form: P(S7xS4) bis}), which completes the proof.
\end{proof}

\begin{rmk}
Case 4 illustrates that the following proposition
is not optimal.
\end{rmk}

\begin{prop}
\label{prop: euler char}
Let $M$ be a compact manifold with nonempty boundary and let $B$ be a closed $b$-dimensional manifold
of nonzero Euler characteristic. If $\max\{2i+7, 3i+4\}<\dim \d M$, then $\dim\ker\pi_i^\Q\big(\iota_{_{M\times B}}\big)\ge
\dim\ker\pi_i^\Q\big(\iota_{_{M}}\big)$.
\end{prop}
\begin{proof}
Consider the following diagram
\[
\xymatrix{
{\mathscr P}(\partial M)\ar[d]_{\delta_\infty} & P(\partial M\times I^b)\ar[l]
\ar[d]_{\delta_b} & P(\partial M)\ar[l]^{\qquad_{\Sigma^b}} 
\ar[r]_{\iota_{_M}} \ar[d]^{^{\times\,\mathbf{id}_B}}& 
\Diff(M)\ar[d]^{^{\times\,\mathbf{id}_B}}  \\
{\mathscr P}(\partial M\times B) &
P(\partial M\times B)\ar[l]\ar[r]^{^{\times\,\chi(B)}} & P(\partial M\times B) \ar[r]^{\i_{_{M\times B}}} &  \Diff(M\times B)
}
\]
where $I^b$ is identified with an embedded disk in $B$ and $\delta_b$ is the extension
by the identity. The middle bottom arrow is the $\chi(B)$-power map
with respect to the group composition. The unlabeled arrows are the
canonical maps into the direct limit, and  $\delta_\infty$
is the stabilization of $\delta_b$.

The rightmost square commutes, while 
the middle one homotopy commutes~\cite[Appendix I]{Hat-surv}. 
Since $\delta_b$ 
homotopy commutes with $\Sigma$, the leftmost square also homotopy
commutes.

It suffices to show that the map ${\times\,\mathbf{id}_B}$ of pseudoisotopy spaces
is \mbox{$\pi_i^\Q$-injective}. 
Since we are in the pseudoisotopy stable range, $\Sigma^b$ and the unlabeled arrows 
are $\pi_i^\Q$-isomorphisms. The $\chi(B)$-power map induces the multiplication by
$\chi(B)$ on the rational homotopy group, see~\cite[Corollary 1.6.10]{Spa-book}. 
Hence the power map is also $\pi_i^\Q$-isomorphism
as $\chi(B)\neq 0$. Finally, $\pi_i^\Q$-injectivity of $\delta_\infty$ follows
because $\mathscr P(-)$ is a homotopy functor and $\delta_\infty$ has a left homotopy inverse
induced by the coordinate projection $\d M\times B\to \d M$. 
\end{proof}

\begin{proof}[Proof of Theorem~\ref{thm-intro: main app}]
By Theorem~\ref{thm-intro: main} it suffices to check that the group
$\ker\pi_{k-1}^{\Q}(\iota_{_{U\times S^m}})$
is nonzero. A lower bound on $\dim\ker\pi_{k-1}^{\Q}(\iota_{_{U\times S^m}})$ 
is given by Theorem~\ref{thm: localize} and we wish to find cases when the bound is positive.

If the sphere bundle associated with the vector bundle with total space $U$ does not
have singly generated rational cohomology, then the lower bound in Theorem~\ref{thm: localize}
can be made arbitrary large by 
Proposition~\ref{prop: signly generated} and Theorem~\ref{thm: bound on pi_i Diff}. This applies
when $U$ is the tangent bundle to $CP^d$, $HP^d$, $d\ge 2$, and the Cayley plane. 

If $U$ is the total space of a vector bundle over $S^{2d}$, $d\ge 2$, 
with nonzero Euler class, then a positive lower bound in Theorem~\ref{thm: localize}
comes from Corollary~\ref{cor: pi=0 implies diff=0} and the part 1 of Theorem~\ref{thm: bounds on P(E) and P(dE)}.
The same argument works to the Hopf $\mathbb R^4$ bundle over $HP^d$ because it has
nonzero Euler class, so the part 2 of Theorem~\ref{thm: bounds on P(E) and P(dE)} applies.

A nontrivial $\mathbb R^3$ over $HP^d$, $d\ge 1$, cannot have a nowhere zero section,
so it must have nonzero Pontryagin class, see~\cite[Theorem V, p.281]{Mas-sph-bund}.
Then a positive lower bound in Theorem~\ref{thm: localize} 
comes from Corollary~\ref{cor: pi=0 implies diff=0} and the part 3 
of Theorem~\ref{thm: bounds on P(E) and P(dE)}.
This applies to the Hopf $\mathbb R^3$ bundle and the bundles in (4).
Finally, (5) follows from Proposition~\ref{prop: euler char}.
\end{proof}

\small
%\bibliographystyle{plain}
%\bibliographystyle{amsalpha}
%\bibliography{mod-pi-revison-no-comments}

\providecommand{\bysame}{\leavevmode\hbox to3em{\hrulefill}\thinspace}
\providecommand{\MR}{\relax\ifhmode\unskip\space\fi MR }
\providecommand{\MRhref}[2]{%
  \href{http://www.ams.org/mathscinet-getitem?mr=#1}{#2}
}
\providecommand{\href}[2]{#2}

\end{document}